\documentclass{elsarticle}
\usepackage[T1]{fontenc}
	\usepackage[full]{textcomp}
	\usepackage{lmodern}
\usepackage{anyfontsize}
\usepackage{mathtools}
\usepackage{amsthm}
\usepackage{mathrsfs}
\usepackage{amssymb}
\usepackage{bm}
\usepackage{caption}
\usepackage{subcaption}
\usepackage[margin=2.5cm]{geometry}
\usepackage{color}
\usepackage{graphicx}
\usepackage{algpseudocode}
\usepackage{float}
\usepackage{url}
\usepackage{enumitem}
\usepackage{amsopn}
\usepackage{tikz}
\usepackage{pgfplots}
\pgfplotsset{compat=newest}
\usepackage{siunitx}
\DeclareSIUnit\angstrom{Å}
\usepackage{booktabs}
\usepackage{hyperref}
\usepackage[capitalize]{cleveref}
\newcommand{\dd}{\mathrm{d}}

\newcommand{\llbrack}{{\lbrack\hspace{-1.25pt}\lbrack}}
\newcommand{\rrbrack}{{\rbrack\hspace{-1.25pt}\rbrack}}
\newcommand{\mean}[1]{\{\!\!\{ #1 \}\!\!\}}
\newcommand{\jump}[1]{\llbrack #1 \rrbrack}

\newcommand{\background}{M}
\makeatletter
\def\@endtheorem{\endtrivlist}%
\makeatother
\newtheorem{remark}{Remark}
\newtheorem{theorem}{Theorem}
\newtheorem{lemma}{Lemma}
\newtheorem{corollary}{Corollary}

\newdefinition{definition}{Definition}

\definecolor{myblue}{HTML}{4477AA}%
\definecolor{myred}{HTML}{EE6677}%
\definecolor{mygreen}{HTML}{228833}%
\definecolor{myyellow}{HTML}{CCBB44}%
\definecolor{mycyan}{HTML}{66CCEE}%
\definecolor{mypurple}{HTML}{AA3377}%
\definecolor{mygrey}{HTML}{BBBBBB}%

\colorlet{color1}{myblue}
\colorlet{color2}{myred}
\colorlet{color3}{mygreen}
\colorlet{color4}{myyellow}
\colorlet{color5}{mycyan}
\colorlet{color6}{mypurple}
\colorlet{color7}{mygrey}

\pgfplotsset{cycle list = {{color1},{color2},{color3},{color4},{color5},{color6}}, samples=100,}
\pgfplotsset{/pgfplots/bar cycle list/.style={/pgfplots/cycle list={{fill=color1},{fill=color2},{fill=color3},{fill=color4},{fill=color5},{fill=color6}}},}
\pgfplotsset{colormap={mybluered}{
    rgb255=(54,75,154)
    rgb255=(74,123,183)
    rgb255=(110,166,205)
    rgb255=(152,202,225)
    rgb255=(194,228,239)
    rgb255=(234,236,204)
    rgb255=(254,218,139)
    rgb255=(253,179,102)
    rgb255=(246,126,75)
    rgb255=(221,61,45)
    rgb255=(165,0,38)
}}
\usepgfplotslibrary{fillbetween}
\usepgfplotslibrary{groupplots}

\makeatletter
\newif\iftag@here

\newcommand*{\taghere}[1][0pt]%
{\ifmeasuring@\else%
  \global\tag@heretrue
  \tikz[remember picture,overlay]{\coordinate (taghere) at (0pt,#1);}%
\fi}

\def\place@tag{%
    \iftagsleft@
      \kern-\tagshift@
      \iftag@here
        \global\tag@herefalse
        \tikz[remember picture,overlay]%
          {\path (taghere) -| node[anchor=base]{\rlap{\boxz@}} (0pt,0pt);}%
      \else
        \shift@tag\row@\relax
            \rlap{\vbox{%
                \normalbaselines
                \boxz@
                \vbox to\lineht@{}%
                \raise@tag
            }}%
                \kern\displaywidth@
      \fi%
    \else
      \kern-\tagshift@
      \iftag@here
        \global\tag@herefalse
        \tikz[remember picture,overlay]%
          {\path  (taghere) -|  node[anchor=base]{\llap{\boxz@}} (0pt,0pt);}%
      \else
        \shift@tag\row@\relax
            \llap{\vtop{%
                \raise@tag
                \normalbaselines
                \setbox\@ne\null
                \dp\@ne\lineht@
                \box\@ne
                \boxz@
            }}%
              \fi%
    \fi
}
\makeatother

\usepackage{soul}
\newlist{linelist}{enumerate*}{1}
\setlist[linelist]{label=(\textbf{C\arabic*})}

\newcommand{\kn}{\mathrm{Kn}}

\begin{document}

\begin{frontmatter}

	\title{Analysis of Moment Closures Using \texorpdfstring{$\varphi$}{phi}-Divergences for Rarefied Dynamics with Binary Collisions and Their Galerkin Discretizations}

	\author[tue,simkinetic]{M.R.A. Abdelmalik}
	\author[oden]{I.M. Gamba}
	\author[tue,simkinetic]{T. Ke{\ss}ler\corref{cor}}
	\ead{torsten.kessler@simkinetic.tech}
	\author[saar]{S. Rjasanow}

	\affiliation[tue]{organization={Department of Mechanical Engineering, Eindhoven University of Technology},
		city={Eindhoven}, country={Netherlands}}
	\affiliation[simkinetic]{organization={Simkinetic B.V.},
		city={Veldhoven}, country={Netherlands}}
	\affiliation[oden]{organization={Oden Institute, University of Texas at Austin},
		city={Austin}, state={Texas}, country={United States}}
	\affiliation[saar]{organization={Department of Mathematics, Saarland University},
	city={Saarbr\"{u}cken}, country={Germany}}
	\cortext[cor]{Corresponding author}

	\begin{abstract}
		This work introduces a robust deterministic framework for approximating solutions of the Boltzmann equation with binary collisions by discretizing their dependence on time, position, and velocity using Galerkin methods. By employing a family of parametric Galerkin closures based on $\varphi$-divergences in velocity space, we derive rigorous hierarchies of moment equations that govern fluid dynamic variables. Addressing the limitation that these closures alone do not guarantee dissipation of a $\varphi$-divergence entropy for the true binary collision operator, we restore this property by formulating a compatible approximate collision operator tailored to each closure. This constructed operator intrinsically retains fundamental physical properties essential for high-fidelity flow simulations, including Galilean invariance, exact conservation of mass, momentum, and energy, and strict dissipation of a $\varphi$-divergence entropy. Furthermore, we show that the resulting closed moment systems are symmetric-dissipative, yielding Cauchy problems that are well-posed locally in time. To translate this mathematical foundation into an efficient computational tool, we discretize the position and time variables with an entropy-stable discontinuous Galerkin (DG) finite element method. The fully implicit, entropy-stable space-time approach enables time steps far beyond typical CFL-limited step sizes and the direct computation of steady states. The robustness and accuracy of the methodology are verified and validated through numerical simulations on the supersonic nozzle flow of argon, mass flow through a channel, and heat transfer between parallel walls, demonstrating agreement with analytical benchmarks, experimental measurements, and stochastic particle simulations.
	\end{abstract}


\end{frontmatter}

\section{Introduction}
The collective interactions among vast numbers of particles are fundamental to understanding transport phenomena in diverse engineering systems. For example, the design of high-altitude spacecraft \cite{ivanov1998}, satellites operating in Earth's orbit \cite{livadiotti2020}, micro-/nano-electromechanical systems \cite{karniadakis2005} and extreme ultraviolet photolithography machines \cite{kerkhof2022,mertens2000} requires simulating the flow of gases in low-pressure conditions. Under such conditions, the ratio of the mean free path of the gas particles to the characteristic physical scales becomes non-negligible, leading to a violation of the continuum hypothesis and the assumption of local thermodynamic equilibrium \cite{chapman1970}. This ratio is quantified by the Knudsen number $\kn$, which delineates the flow regimes: the continuum regime ($\kn \ll 1$), the transitional regime ($\kn \sim 1$), and the free-molecular regime ($\kn \to \infty$). Such a violation implies that classical fluid dynamic descriptions, such as the Euler and Navier--Stokes--Fourier equations, fail to capture the physical flow fields, necessitating a description of the flow based on particle transport.

Kinetic theory \cite{villani2002} offers a unified framework for modeling rarefied gases as systems composed of an effectively infinite number of particles. The state of the (monatomic) gas is described by a distribution function of time $t$, position $\bm x$, and velocity $\bm v$,
\begin{equation}\label{eq:densitydistribution}
	f : (0, \infty) \times \Omega \subset \mathbb{R}^3 \times \mathbb R^3 \to (0, \infty),~(t, \bm x, \bm v) \mapsto f(t, \bm x, \bm v),
\end{equation}
whose evolution is governed by the Boltzmann equation,
\begin{equation}
	\partial_t f + \bm v \cdot \bm \nabla_{\bm x} f = \mathcal C(f),
\end{equation}
with Boltzmann's binary collision operator $\mathcal C$ given in~\eqref{eq:collision} below.
The distribution encodes the fluid-dynamic fields as velocity moments of $f$: the density $\rho = \int_{\mathbb R^3} f \, \dd \bm v$, the bulk velocity $\bm u$, the temperature $\theta$, and the heat flux $\bm q$.
The collision operator $\mathcal C$ arises as the collective effect of binary collisions among a divergent number of particles under Boltzmann--Grad scaling \cite{grad1958},
\begin{equation}\label{eq:collision}
	\mathcal{C}(f(t,\bm x, \cdot))(\bm v) =\int_{\mathbb{R}^3}\int_{\mathbb{S}^2}
	B(\bm v - \bm v_*, \bm \sigma)
	\Big(f(\bm{\acute v}_*)f(\bm{\acute v}) - f(\bm v_*)f({\bm v})\Big) \, \dd{\bm\sigma} \, \dd \bm v_*,
\end{equation}
where $\mathbb S^2$ denotes the unit sphere in $\mathbb R^3$, and the post-collisional velocities $\bm{\acute v}$ and $\bm{\acute v_*}$ are related to the pre-collisional velocities $\bm v$ and $\bm v_*$ by
\begin{subequations}\label{eq:post-collional-velocities}
	\begin{align}
		\bm{\acute v}   & = \frac{1}{2}\big(\bm v + \bm v_* \big) + \frac{1}{2} |\bm v - \bm v_*| \bm \sigma, \\
		\bm{\acute v_*} & =\frac{1}{2}\big(\bm v + \bm v_* \big) - \frac{1}{2} |\bm v - \bm v_*| \bm \sigma,
	\end{align}
\end{subequations}
with $\bm \sigma \in \mathbb S^2$ the scattering direction. The collision kernel $B$ depends on the modeled inter-particle interaction; we assume the typical form for power-law interactions,
\begin{equation}\label{collisionWeightVHS}
	B(\bm v - \bm v_*, \bm \sigma) = |\bm v - \bm v_*|^{\lambda}  b\big( \bm{\hat u} \cdot \bm \sigma \big), \quad \lambda \in (-3, 1],
\end{equation}
with an integrable function $b : [-1,1] \to \mathbb R$ and $\bm{\hat u} = (\bm v - \bm v_*)/|\bm v - \bm v_*|$ the normalized relative velocity. This choice includes hard spheres ($\lambda = 1$) and Maxwell's molecules ($\lambda = 0$); the limit $\lambda \to -3$ recovers the Rutherford cross section for the Coulomb interaction.

While the collision operator in \eqref{eq:collision} describes the rate of change, due to collisions, of distributions that are generally out of thermodynamic equilibrium, $\mathcal C$ also encodes a characterization of thermodynamic equilibrium states: detailed balance of pre- and post-collisional states in the integrand of $\mathcal C$ characterizes the Maxwellian distributions,
\begin{equation}\label{eq:colleq}
	f(\acute{\bm v}_*)f(\acute{\bm v}) = f(\bm v_*)f(\bm v) \Leftrightarrow f =
	M_{\rho, \bm u, \theta}(\bm v) :=
	\frac{\rho}{(2\pi \theta)^{\frac{3}{2}}}\exp\left(-\frac{|\bm v-\bm u|^2}{2 \theta}\right)
\end{equation}
for some $(\rho, {\bm u}, \theta)\in\mathbb{R}_{>0}\times\mathbb{R}^3\times\mathbb{R}_{>0}$. Underpinning the characterization in \eqref{eq:colleq} is the conservation of particle mass, momentum and energy across collisions \cite{cercignani1988}:
\begin{equation} \label{eq:particle_conserv}
	\psi(\bm v) + \psi(\bm v_*) = \psi(\acute{\bm v}) + \psi(\acute{\bm v}_*) \Leftrightarrow \psi(\bm v)\in\mathrm{span}\{1, v_1, v_2, v_3, |\bm v|^2\} =:\mathscr I.
\end{equation}
The binary collision operator $\mathcal C$ thus describes collisions within particle systems both in and out of thermodynamic equilibrium, the perspective at the core of this work.

The equilibrium characterization \eqref{eq:colleq} embedded in the binary collision operator gives rise to thermodynamic relations for the fluid dynamical description induced by moments of the collision operator against a (bounded and measurable) test function $\psi(\bm v)$,
\begin{equation}\label{eq:mom_coll}
	\int_{\mathbb{R}^3} \psi(\bm v)\mathcal{C}(f) \, \dd\bm v
	= \int_{\mathbb{R}^3} \int_{\mathbb{R}^3}\int_{\mathbb{S}^2}
	\psi(\bm v)
	\Big(f(\bm{\acute v}_*)f(\bm{\acute v}) - f(\bm v_*)f({\bm v})\Big)  B(\bm v - \bm v_*, \bm \sigma)\, \dd{\bm\sigma} \, \dd \bm v_* \, \dd \bm v.
\end{equation}
The measure $ B(\bm v - \bm v_*, \bm \sigma) \dd{\bm\sigma} \, \dd \bm v_* \, \dd \bm v$ is invariant under the transformations
\begin{equation} \label{eq:invar_trans}
	(\bm \sigma, \bm v_*, \bm v) \mapsto  (\bm \sigma, \bm v, \bm v_*), \ (\bm \sigma, \acute{\bm v}, \acute{\bm v}_*) \mapsto (\bm \sigma, \bm v, \bm v_*) \text{ and }  (\bm \sigma, \acute{\bm v}_*, \acute{\bm v}) \mapsto (\bm \sigma, \acute{\bm v}, \acute{\bm v}_*),
\end{equation}
which leads to Galilean symmetry of $\mathcal C$ under translations $\mathcal T_{\bm z} f (\bm v) = f(\bm v-\bm z)$, $\bm z\in\mathbb{R}^3$, and rotations $\mathcal T_{\bm O} f = f(\bm O^{\top} \bm v)$, $\bm O\in\mathbb{R}^{3 \times 3}$ orthogonal,
\begin{equation}\label{eq:galil_symm}
	\int_{\mathbb R^3}
	\mathcal{T}_{\bm z}\mathcal{T}_{\bm O}\mathcal{T}_{-\bm z}\psi(\bm v) \mathcal C(f)
	\, \dd \bm v
	=
	\int_{\mathbb R^3} \psi(\bm v) (\mathcal{T}_{\bm z}\mathcal{T}_{\bm O}\mathcal{T}_{-\bm z})^* \mathcal C(f)
	\, \dd \bm v
	=
	\int_{\mathbb R^3} \psi(\bm v) \mathcal C((\mathcal{T}_{\bm z}\mathcal{T}_{\bm O}\mathcal{T}_{-\bm z})^* f)
	\, \dd \bm v.
\end{equation}
Applying the transformations \eqref{eq:invar_trans} to \eqref{eq:mom_coll} yields the symmetrized form
\begin{equation}\label{eq:symm_coll}
	\int_{\mathbb{R}^3} \psi(\bm v)\mathcal{C}(f) \, \dd\bm v
	= -\frac{1}{4} \int_{\mathbb R^3} \int_{\mathbb R^3} \int_{\mathbb S^2}
	(\acute f \acute f_* - f f_*)
	(\acute \psi + \acute \psi_* - \psi - \psi_*
	\Big)
	B\, \dd \bm \sigma \, \dd \bm v \, \dd \bm v_*,
\end{equation}
where, for brevity, we write $\acute{g}$, $g_*$ and $\acute{g}_*$ for $g(\acute{\bm v})$, $g(\bm v_*)$ and $g(\acute{\bm v}_*)$, respectively. Using \eqref{eq:particle_conserv} in \eqref{eq:symm_coll}, it follows directly that
\begin{equation}\label{eq:CollInvariant}
	\int_{\mathbb{R}^3}
	\psi(\bm v)\,\mathcal{C}(f)
	\, \dd \bm v
	=0
	\Leftrightarrow
	\psi(\bm v) \in \mathscr{I}.
\end{equation}
In particular, the choices $\psi = 1$, $\bm v$ and $|\bm v|^2/2$ in \eqref{eq:CollInvariant} imply the conservation of fluid mass, momentum and energy across collisions. Setting $\psi(\bm v) = \eta'(f)$, with
\begin{equation}\label{eq:boltz_ent}
	\eta(f)=f \ln(f) - f
\end{equation}
the Boltzmann entropy function, yields the entropy production inequality
\begin{equation}\label{eq:ent_prod_coll}
	\int_{\mathbb{R}^3}
	\eta'(f) \,\mathcal{C}(f)
	\, \dd \bm v
	=
	\int_{\mathbb{R}^3}
	\ln(f) \,\mathcal{C}(f)
	\, \dd \bm v
	=  -\frac{1}{4} \int_{\mathbb R^3} \int_{\mathbb R^3} \int_{\mathbb S^2}
	(\acute f \acute f_* - f f_*)
	\ln \left(\frac{\acute f \acute f_*}{ff_*}\right)
	B\, \dd \bm \sigma \, \dd \bm v \, \dd \bm v_* \leq 0,
\end{equation}
with equality if and only if $\ln f \in \mathscr I$. Together with \eqref{eq:colleq}, the entropy production inequality yields the characterization of local thermodynamic equilibrium
\begin{equation}
	\label{eq:Equilibrium}
	f = M_{\rho,\bm u,\theta}(\bm v)
	\iff
	\int_{\mathbb{R}^3} \ln(f)\,\mathcal{C}(f)\, \dd \bm v  = 0
	\iff
	\mathcal{C}(f) = 0.
\end{equation}
The Galilean symmetry \eqref{eq:galil_symm} and the relaxation induced by \eqref{eq:ent_prod_coll} towards the local equilibria \eqref{eq:Equilibrium} underpin the consistency of the kinetic description with the Euler and Navier--Stokes--Fourier descriptions in the continuum and near-continuum regimes, respectively \cite{saint-raymond2009}. Although other operators satisfy the conservation identity \eqref{eq:CollInvariant} and the entropy production inequality \eqref{eq:ent_prod_coll}, such as the BGK \cite{bhatnagar1954} and Shakhov \cite{shakhov1968} operators, the binary collision operator \eqref{eq:collision} is distinguished by its emergence from the collective behavior of particle systems. Kinetic systems governed by \eqref{eq:collision} are therefore uniquely suited to describe rarefied dynamics in the transitional regime between the continuum and free-streaming limits, by virtue of their inherent characterization of relaxation towards local thermodynamic equilibrium.

While the high dimensionality of the distribution function \eqref{eq:densitydistribution} already poses a challenge for numerical approximations of kinetic systems, accounting for binary collisions \eqref{eq:collision} compounds the difficulty, as evaluating the nested integrals over $\mathbb{R}^3 \times \mathbb{S}^2$ is computationally expensive \cite{dimarco2014}. The simulation of binary collisions has therefore been predominantly based on particle methods such as Direct Simulation Monte Carlo (DSMC) \cite{bird:1970}, which replaces the deterministic evaluation of the collision integrals by stochastic sampling over particles and whose convergence is well established \cite{wagner1992}. At the same time, the computational cost of DSMC grows considerably in slow flows and in regimes approaching the fluid-dynamical limit, where stochastic fluctuations are large compared to the fluid velocity, and the inherent stochastic noise complicates stable coupling to continuum descriptions such as the Navier--Stokes--Fourier equations \cite{degond2005}. These considerations highlight the need for methods that resolve rarefied dynamics while bridging to the fluid dynamical regime.

Deterministic methods may offer a noise-free alternative to DSMC for kinetic equations. The predominant approaches are Fourier transform methods (FTM), which rest on the Fourier representation of the elastic collision operator \eqref{eq:collision} for Maxwell-type interactions \cite{bobylev1988,bobylev2020}, whose transition probability retains the angular element but is independent of the intermolecular potential. The first conservative schemes for hard spheres in three dimensions \cite{bobylev1997,bobylev1999} apply the fast Fourier transform to the collision operator together with an algorithm that enforces the collision invariants. Spectral-Lagrangian methods for non-equilibrium statistical states \cite{gamba2009,gamba2010} restore the collision invariants by an $L^2(\mathbb R^d)$ minimization at $O(N)$ cost per time step for any velocity dimension $d \ge 3$. The same construction covers elastic and inelastic collisions, the latter of which do not conserve energy. For hard potentials in the space-homogeneous case, this Lagrangian-based conservative spectral method admits convergence and error estimates, together with a conditional approximation of the Maxwellian equilibrium determined by the initial data \cite{alonso2018}.

The $L^2$ minimization has since been adopted as an $O(N)$ correction to schemes that do not enforce the collision invariants, among them solvers for the Landau equation \cite{zhang2017,pennie2022}, for gas mixtures with anisotropic scattering \cite{gamba2014}, for polyatomic gases with internal energy levels \cite{munafo2014}, for kinetic swarming models \cite{gamba2015}, and for BGK, discontinuous Galerkin and moment schemes \cite{dimarco2013,zhang2018,dimarco2018}. Spectral methods that do not enforce the collision invariants \cite{pareschi1996,pareschi2000}, including the fast methods for general collision kernels \cite{mouhot2006,gamba2017}, have incorporated the same minimization, as has a moment preserving Fourier--Galerkin method \cite{pareschi2022}. Conservation of the collision invariants notwithstanding, discrete counterparts of the Galilean symmetry \eqref{eq:galil_symm} and of the entropy production relation \eqref{eq:ent_prod_coll} are in general not inherited, which has motivated steady-state preserving variants \cite{filbet2015}.

Discrete velocity models (DVM) \cite{mieussens2000,aristov2001} form a second family of deterministic schemes; like FTM, they truncate the velocity domain of the distribution function \eqref{eq:densitydistribution}. The so-called normal DVM \cite{bobylev1995} restores discrete analogues of these structural properties. Nevertheless, artifacts of the arbitrary velocity truncation and of misalignment between the normal vectors of the spatial domain boundary and the velocity grid persist and may impede the imposition of inflow boundary conditions \cite{brull2014}.

The method of moments \cite{grad1949,levermore1996,abdelmalik2016a,torrilhon2016} has emerged as an efficient deterministic approach that approximates solutions of kinetic systems by identifying a finite set of parameters of an approximate distribution based on its moments in velocity space. It is designed to approximate distributions over the full velocity domain $\mathbb R^3$, thereby avoiding the truncation-related artifacts introduced by DVM and FTM. Applied to kinetic equations, the method of moments engenders systems of evolution partial differential equations for the moments of the distribution function; these systems may be solved implicitly in time and designed to recover fluid dynamic descriptions in the continuum limit. Therefore, the moment-based approximation enables the use of classical coupling strategies for multiphysics \cite{keyes2013} and multiscale \cite{weinan2003} problems. However, previously derived estimates \cite{junk1998} and limits \cite{bardos1991} indicate that both the stability of implicit-in-time approximations and recovery of the fluid-dynamical limits are contingent upon methods that retain finite-dimensional representations of the Galilean symmetry \eqref{eq:galil_symm}, entropy production \eqref{eq:ent_prod_coll} and equilibrium \eqref{eq:Equilibrium} relations. While linear variants of the moment method have been applied to kinetic systems with binary collisions \eqref{eq:collision} in restricted settings \cite{cai2013}, such methods may fail to preserve the entropy production inequality \eqref{eq:ent_prod_coll}, which may preclude numerical approximations that are stable and consistent with fluid dynamic limits.

In this work we consider the application of moment methods that construct finite-dimensional approximations of distributions, referred to as \textit{moment closures}, conforming to minimizers of a class of generalized (relative) entropy functions, known as $\varphi$-divergences \cite{abdelmalik2016a}, subject to moment constraints. A $\varphi$-divergence is the convex functional
\begin{equation}
	D_\varphi(f \| M) = \int_{\mathbb R^3} M \, \varphi(f / M) \, \dd \bm v,
\end{equation}
that measures the discrepancy between the distribution $f$ and a reference Maxwellian $M$ through a strictly convex generator $\varphi$, the Kullback--Leibler divergence,
$\varphi(z) = z \ln z - z + 1$, being the canonical example. Such moment closures provide a unique framework for deriving finite-dimensional representations of kinetic equations that conform to symmetric hyperbolic systems, retaining the Galilean symmetries \eqref{eq:galil_symm} and equilibrium relations \eqref{eq:CollInvariant}. %
Specific to this work is the choice of a family of entropy functions that approximate Boltzmann's entropy function \eqref{eq:boltz_ent}. However, such approximations of the entropy function do not guarantee the entropy production inequality \eqref{eq:ent_prod_coll} for the binary collision operator \eqref{eq:collision}. In the remainder of this work we establish that moment closures based on $\varphi$-divergences are approximate entropy functions for kinetic equations with binary collisions \eqref{eq:collision}. Our characterization of approximate $\varphi$-divergence entropies for binary collision operators \eqref{eq:collision} constitutes the main contribution of this work, in particular:
\begin{enumerate}
	\item we construct a family of approximations to the binary collision operator \eqref{eq:collision} associated with the family of $\varphi$-divergence approximations of Boltzmann's entropy \eqref{eq:boltz_ent}. We establish that each member of the approximate collision operator family
	      \begin{enumerate}
		      \item satisfies \eqref{eq:invar_trans}, thereby complying with the physical principles of objectivity,
		      \item satisfies \eqref{eq:CollInvariant}, thereby retaining the collision invariants and not introducing any additional invariants,
		      \item satisfies \eqref{eq:ent_prod_coll} where $\eta'(f)$ is replaced with the derivative of the approximate $\varphi$-divergence entropy,
		      \item satisfies the characterization of equilibrium \eqref{eq:Equilibrium} for $\varphi$-divergence entropies with associated approximations of the Maxwellians $M_{\rho, \bm u, \theta}$;
	      \end{enumerate}
	\item we prove the well-posedness of the $\varphi$-divergence optimization problem subject to moment constraints in subspaces of $L^{1}(\mathbb{R}^3)$;
	\item we establish the finite-time well-posedness of the nonlinear Cauchy initial value problem of the partial differential equations engendered by the moment closure. The proof of well-posedness follows from the characterization of solutions of the so-called \textit{symmetric-dissipative} systems \cite{kawashima2004};
	\item we formulate a fully implicit, entropy-stable space-time discontinuous Galerkin finite element discretization of the closed moment systems, for which the discrete approximate entropy is bounded by that of the initial data (\cref{them:entropy-stability});
	\item we verify and validate our $\varphi$-divergence-based numerical approximations of solutions of kinetic equations with binary collisions against analytical benchmarks and experimental data, and against DSMC reference solutions for the supersonic nozzle flow of argon.
\end{enumerate}

The remainder of this paper is organized as follows. \Cref{sec:BE} reviews the salient properties of the Boltzmann equation, \cref{sec:approx_coll} introduces the converging sequence of approximate collision operators, and \cref{sec:MomSys} establishes the finite-time well-posedness of the closed moment systems. \Cref{sec:DGFEM} presents the fully implicit space-time discontinuous Galerkin discretization and its entropy stability. \Cref{sec:numerical-results} presents the verification and validation studies, and \cref{sec:outlook} concludes.

\section{Properties of Solutions of the Boltzmann Equation and a Property-Preserving Method of Moments}
\label{sec:BE}
Consider a monatomic gas inside a bounded spatial domain $\Omega \subseteq \mathbb R^3$ with inflow and outflow parts of $\partial\Omega$ respectively defined as
\begin{equation}
	\mathbb R^3_\text{in}(\bm x) = \{\bm v \in \mathbb R^3: \bm v \cdot \bm n(\bm x) \leq 0 \}, \quad
	\mathbb R^3_\text{out}(\bm x) = \{\bm v \in \mathbb R^3: \bm v \cdot \bm n(\bm x) \geq 0 \},
\end{equation}
where $\bm x \in \partial \Omega$ and $\bm n(\bm x)$ is the outward normal vector of $\Omega$ at $\bm x$.
We describe the state of the fluid by its kinetic distribution function
and assume that $f$ is governed by the Boltzmann equation
\begin{subequations}\label{eq:ibvBE}
	\begin{align}\label{eq:Boltzmann}
		\partial_t f + \bm v \cdot \bm \nabla_{\bm x} f - \mathcal{C}(f)     & =0, \quad t>0,~\bm x\in\Omega, ~\bm v\in\mathbb{R}^3;                                                \\
		f(0,\bm x, \bm v) - f_0(\bm x,\bm v)                                 & =0, \quad \bm x\in\Omega, ~ \bm v\in\mathbb{R}^3;                                                    \\
		\bm v \cdot \bm n(\bm x) (f(t, \bm x, \bm v) - f_w(t, \bm x, \bm v)) & =0, \quad t \geq 0,~\bm x \in \partial \Omega,~ \bm v \in \mathbb R^3_\text{in}(\bm x).\label{eq:bc}
	\end{align}
\end{subequations}
In the remainder of this section we highlight salient properties of solutions of \eqref{eq:Boltzmann} that are to be retained or approximated by our proposed moment method.

To complete the description of the initial--boundary value problem in \eqref{eq:ibvBE}, we specify the types of interaction of the fluid in $\Omega$ with the surrounding space. We consider the so-called specular reflection boundary condition
\begin{equation}\label{eq:bc-specular}
	f_w(t, \bm x, \bm v) = f(t, \bm x, \bm v - 2 \bm v \cdot \bm n(\bm x) \bm n(\bm x)), \quad t \geq 0,~\bm x \in \partial \Omega,~ \bm v \in \mathbb R^3_\text{in}(\bm x),
\end{equation}
where the impinging particles are elastically reflected, and the diffuse reflection boundary condition
\begin{equation}\label{eq:bc-diffuse}
	f_w(t, \bm x, \bm v) = \rho_\text{out}(t, \bm x) M_{1,\bm u_w(t, \bm x), \theta_w(t, \bm x)} \quad \forall t>0,~ \bm x \in \partial \Omega,~ \bm v \in \mathbb R^3_\text{in}(\bm x),
\end{equation}
where impinging particles are thermalized with temperature $\theta_w > 0$ and reflected with some velocity drawn from Maxwellian density $M_{1,\bm u_w, \theta_w}$ with $\bm u_w \in \mathbb R^3$ satisfying $\bm u_w\cdot\bm n(\bm x) = 0$, and the scaling factor $\rho_\text{out}$ enforces the balance of mass-flux at the boundary $\partial\Omega$,
\begin{equation}\label{eq:rho-out}
	\int_{\mathbb R^3_\text{out}(\bm x)} |\bm v \cdot \bm n(\bm x)| f(t, \bm x, \bm v) \, \dd \bm v%
	= \rho_{\text{out}}
	\displaystyle \int_{\mathbb R^3_\text{in}(\bm x)} |\bm v \cdot \bm n(\bm x)| M_{1,\bm u_w, \theta_w} \, \dd \bm v.
\end{equation}

Solutions of \eqref{eq:ibvBE} possess symmetry, conservation and dissipation properties on account of fundamental properties of $f\mapsto\mathcal{C}(f)$ and its corresponding weak form given by \eqref{eq:mom_coll} or \eqref{eq:symm_coll}, where we have implicitly assumed a specific domain of definition for $\mathcal C$,
\begin{equation}\label{eq:collisionDomainDefinition}
	\mathscr D(\mathcal C) = \{f \in L^1_2(\mathbb R^3) : \mathcal C(f) \in L_k^{1}(\mathbb R^3) \, \forall k \in \mathbb N \}.
\end{equation}
Here $L^1_k(\mathbb R^3)$ denotes the $L^1$ space with $k$-th order velocity weight $(1 + |\bm v|^2)^{k/2}$;
our choice of $\mathscr D(\mathcal C)$ ensures that all polynomials are admissible test functions.
For hard potentials, $\lambda \in [0,1]$, $\mathcal C$ maps Schwartz functions to Schwartz functions~\citep{duduchava2005},
so $S(\mathbb R^3) \subseteq \mathscr D(\mathcal C)$, which is therefore dense in $L^1(\mathbb R^3)$.

To approximate weak solutions of \eqref{eq:ibvBE}, we consider a semi-weak representation in the velocity variable:
integrating \eqref{eq:ibvBE} against test functions $\psi \in \mathscr O$,
the space of locally integrable functions in $\mathbb R^3$ with at most polynomial growth,
yields the moment system for $f \in \mathscr D(\mathcal C)$, one equation for each $\psi \in \mathscr O$,
\begin{equation}\label{eq:moments}
	\partial_t \int_{\mathbb R^3}
	\psi(\bm v) f(t, \bm x, \bm v)  \, \dd \bm v
	+
	\nabla_{\bm x} \cdot \int_{\mathbb R^3} \bm v \psi(\bm v)
	f(t, \bm x, \bm v) \, \dd \bm v
	-
	\int_{\mathbb R^3} \psi(\bm v)
	\mathcal C\big(f(t, \bm x, \cdot)\big)(\bm v)
	=
	0,
\end{equation}
for $t > 0$ and $\bm x \in \Omega$, subject to appropriate initial and boundary conditions.

In the remainder of this section we review these properties, viz., invariance under Galilean transformations, conservation of mass, momentum and energy, and dissipation of appropriate entropy functions. Our treatment is standard (see, for instance,~\cite{cercignani1988}) and is included for completeness.

\subsection{Galilean Symmetry}
For $\bm u \in \mathbb R^3$ and orthogonal $\bm O \in \mathbb R^{3 \times 3}$, the Galilei transform
$\big( \mathcal G_{\bm u, \bm O} f \big)(t, \bm x, \bm v) = f(t, \bm O\bm x - t \bm u, \bm O \bm v - \bm u)$
acts on functions on time-phase space.
As a consequence of the Galilean symmetry \eqref{eq:galil_symm} of the collision operator
and the chain rule applied to the transport operator, the Boltzmann operator
$\mathcal B(f) = \partial_t f + \bm v \cdot \nabla_{\bm x} f - \mathcal C(f)$
commutes with all Galilei transforms, so that $\mathcal G_{\bm u, \bm O} f$ solves the Boltzmann equation
whenever $f$ does and, provided solutions are unique, $f$ inherits the invariances of its data.
In \Cref{sec:numerical-results} we exploit such invariances, including the boundary conditions,
to reduce the functional dependency of the distribution function, for instance the cylindrical
symmetry in the computation of heat fluxes.

\subsection{Conservation Properties}
Testing \eqref{eq:Boltzmann} against each collision invariant $\psi \in \mathscr I$, the identity \eqref{eq:CollInvariant} associates with \eqref{eq:Boltzmann} the local conservation law
\begin{equation}
	\label{eq:locConsvLaw}
	\partial_t
	\int_{\mathbb{R}^3} \psi f\, \dd \bm v
	+
	\sum_{i=1}^3
	\partial_{x_i}
	\int_{\mathbb{R}^3}
	v_i \psi f \, \dd \bm v
	= 0,
\end{equation}
so that the local mass, momentum and energy are conserved.
Integrating \eqref{eq:locConsvLaw} over $\Omega$ and applying the divergence theorem
shows that specular reflection \eqref{eq:bc-specular} conserves global mass and energy
(since $|\bm v - 2(\bm v\cdot\bm n)\bm n|^2 = |\bm v|^2$) but not momentum,
while diffuse reflection \eqref{eq:bc-diffuse} conserves only global mass,
on account of \eqref{eq:rho-out}.

\subsection{Entropy Production}\label{sec:dissipation-continuous}
For any Maxwellian $M$ that is independent of $x$ and $t$, testing \eqref{eq:Boltzmann} against $\ln(f/M)$ yields the local relative-entropy law relative to $M$ for weak solutions of \eqref{eq:Boltzmann}:
\begin{align}
	\label{eq:relEntDiss}
	\partial_t\int_{\mathbb{R}^3} f \ln(f/M) - f + M \, \dd \bm v+\sum_{i=1}^3 \partial_{x_i}\int_{\mathbb{R}^3} v_i (f \ln(f/M) - f + M) \, \dd \bm v = \int_{\mathbb{R}^3} \ln(f/M)\mathcal{C}(f)\, \dd \bm v \leq 0\,.
\end{align}
Integrating \eqref{eq:relEntDiss} over $\Omega$, applying the divergence theorem,
and controlling the sign of the boundary term via the Darrozès--Guiraud
inequality~\citep{cercignani1988} (\cref{lem:Darrozes_Guiraud} below) yields the
global relative-entropy dissipation law
\begin{align}\label{eq:global-entropy-dissipation}
	\partial_t \int_{\Omega\times\mathbb{R}^3} f \ln(f/M) - f + M \, \dd\bm v \, \dd\bm x \leq 0.
\end{align}
\begin{lemma}\label{lem:Darrozes_Guiraud}
	Let $\widetilde{\eta}: \mathbb R \to \mathbb R$ be a convex function and impose at $\bm x \in \Gamma$ a convex combination of specular and diffuse reflection with zero wall velocity.
	If $M$ is a solution to the boundary condition at $\bm x$, then
	\begin{equation} \label{eq:lemmDG}
		\bm n(\bm x) \cdot \int_{\mathbb{R}^3} \bm v M \widetilde{\eta}(f / M)\, \dd\bm v\geq 0.
	\end{equation}
\end{lemma}
\begin{remark}\label{rem:maxwellian-bc}
	The Maxwellian $M$ satisfies a convex combination of specular~\eqref{eq:bc-specular}
	and diffuse~\eqref{eq:bc-diffuse} reflection if $M$ is proportional to the wall Maxwellian;
	for pure specular reflection it suffices to choose $M$ with vanishing normal velocity,
	$\int_{\mathbb R^3} \bm n(\bm x) \cdot \bm v M(\bm v) \, \dd \bm v = 0$.
\end{remark}

\subsection{Property-Preserving Method of Moments}
Closing the system of moment equations \eqref{eq:moments} requires restricting the distribution function $f$ to a finite-dimensional space. To retain the local conservation laws \eqref{eq:locConsvLaw} in \eqref{eq:moments}, it suffices to require that $\mathscr{I}\subset\mathscr{O}$. Furthermore, the preservation of Galilean invariance relations in~\eqref{eq:moments} requires that the test space $\mathscr O$ and the trial space are both invariant with respect to $\mathcal G_{\bm u, \bm O}$. Such invariance is guaranteed by choosing polynomial test and trial spaces \cite{junk2002}.

To preserve the entropy inequality \eqref{eq:relEntDiss} we determine the approximate distribution by an entropy-minimization principle, using maximum entropy~\cite{dreyer1987,levermore1996} or, more generally,
$\varphi$-divergences~\cite{ali1966,csiszar1972,abdelmalik2016a}.
In this case, we seek an approximation $f_\text{KL}$ to a distribution function $f$ by minimizing the $\varphi$-divergence,
\begin{equation}
	\min \int_{\mathbb R^3} M \varphi(f_\text{KL} / M) \, \dd \bm v
	\quad \text{s.\,t.} \quad \forall m \in \mathscr M: \int_{\mathbb R^3} m \, f_\text{KL} \, \dd \bm v = \int_{\mathbb R^3} m \, f \, \dd \bm v,
\end{equation}
with a strictly convex function $\varphi$ of at most polynomial growth at infinity and a prior distribution $M$, usually a fixed Maxwellian;
here $\mathscr M$ denotes a finite-dimensional space of polynomial moment (test) functions, spanned by a basis $\bm m$ and containing the collision invariants $\mathscr I$. In words, $f_\text{KL}$ is the distribution closest to the prior $M$, as measured by the $\varphi$-divergence, that reproduces the prescribed moments of $f$.
First-order optimality conditions for the associated Lagrangian yield
\begin{equation}
	\varphi'(f_\text{KL} / M) = \bm \lambda \cdot \bm m
\end{equation}
pointwise, with Lagrange multipliers $\bm \lambda$ associated with the moment constraints. Since $\varphi$ is strictly convex, its derivative $\varphi'$ is injective; assuming that $\varphi'$ is also surjective, we may invert it to obtain the renormalized form
\begin{equation}
	f_\text{KL} = \beta(\bm \lambda \cdot \bm m),
	\qquad
	\beta(g) = M \varphi'^{-1}(g),
\end{equation}
with the renormalization mapping $\beta$. Inserting this form of $f_\text{KL}$ back into the moment constraints yields a nonlinear system for the multipliers,
\begin{equation}
	\int_{\mathbb R^3} \bm m \, \beta(\bm \lambda \cdot \bm m) \, \dd \bm v = \int_{\mathbb R^3} \bm m \, f \, \dd \bm v.
\end{equation}

With $\varphi(z) = z \ln z - z$, we have $\varphi'(z) = \ln z$
and thus recover the exponential renormalization $\beta(g) = e^g$
for a constant prior distribution, $M = 1$.
While this choice ensures that $\ln f_\text{KL} \in \mathscr M$,
it was shown in \cite{junk1998,junk2002} that the resulting closure may be ill-posed:
if $\mathscr M$ contains polynomials of even degree larger than two, then every neighborhood of equilibrium contains realizable moment vectors, that is, moments of some nonnegative distribution, that cannot be represented by any exponential ansatz $e^{\bm \lambda \cdot \bm m}$, so that the constrained minimization problem fails to have a solution. Consequently, when $\mathscr{M}$ contains polynomials of super-quadratic growth, little can be said
about the well-posedness of the initial-value problem even near equilibrium,
in accordance with well-posedness theories for hyperbolic laws~\cite{majda:1984wj}.
Moreover, the transport term in the Boltzmann equation may be singular at equilibrium points \cite{junk1998,dreyer2001}.

To address these ill-posedness issues associated with entropy-based moment closures,
we consider a sequence of renormalization mappings $\beta_N$, $N \in \mathbb N$, proposed in~\cite{abdelmalik2016a},
\begin{equation}\label{eq:phi_map}
	\beta_N(g) = \background \left( 1 + \frac{g}{N} \right)^N,
\end{equation}
with a fixed, global Maxwellian $M$.
Since $\beta_N(g)(\bm v) \to M(\bm v) e^{g(\bm v)}$ as $N \to \infty$, the mappings $\beta_N$ recover, up to a shift of $g$ by the quadratic polynomial $\log M$, the exponential renormalization map in the limit.

For $N = 1$, the renormalization mapping \eqref{eq:phi_map} coincides with
that of the Galerkin--Petrov method~\cite{gamba2018,kessler2019,hanke2023}
and with that of Grad \cite{grad1949,grad1958} near $M$.
The associated family of $\varphi$-divergences is given by
\begin{equation}\label{eq:phiNkullb}
	\varphi_N\left(z\right) = z N \left(z^{\frac{1}{N}}\frac{N}{N+1} - 1 \right) + \frac{N}{N+1},%
\end{equation}
with the limits
\begin{equation}\label{eq:convergence_phi_n}
	\lim_{N\to\infty} \varphi'_N(z) = \ln z, \quad \lim_{N \to \infty} \varphi_N(z) =  z\ln z- z + 1,
\end{equation}
for all $z > 0$.
In contrast to the exponential closure, the reconstruction problem for the coefficients $\bm \lambda$,
\begin{equation}\label{eq:phi_closure}
	\int_{\mathbb R^3} \bm m \, \beta_N(\bm \lambda \cdot \bm m) \, \dd \bm v = \int_{\mathbb R^3} \bm m \, f \, \dd \bm v,
\end{equation}
has a unique solution, given mild assumptions on the data $f$.
\begin{theorem}\label{thm:reconstruction-problem}
	For $f \in L^1_K(\mathbb R^3)$, where $K$ denotes the maximal polynomial degree of $\bm m$, the reconstruction problem
	has a unique solution for odd $N$.
\end{theorem}
However, unlike $\ln f$, the function $\varphi_N'(f/M)$ may not satisfy the dissipation relation; in particular, the dissipative properties of
\begin{equation}\label{eq:phidiss}
	\int_{\mathbb R^3} \varphi_N'\left(\frac{f(t,\bm x,\bm v)}{M(\bm v)}\right) \mathcal C(f(t,\bm x, \cdot))(\bm v) \, \text{d}\bm v
\end{equation}
are not well understood, leaving a gap in our understanding of the moment systems obtained by applying
the $\varphi$-divergence closures in \eqref{eq:phi_map} to \eqref{eq:moments}.
Such dissipative properties play an important role in the design of stable space-time numerical approximations,
see~\cite{barth2006,abdelmalik2016b} and~\cite{hughes1986,tadmor1987}.
In the next section, we elucidate the dissipative properties of~\eqref{eq:phidiss}
by constructing a sequence of property-preserving approximations $\mathcal C_N$ to Boltzmann's collision operator $\mathcal C$.

\section{\texorpdfstring{$\varphi$}{phi}-Divergence Conforming Approximations of the Collision Operator}
\label{sec:approx_coll}

In this section, we control the error in the entropy production inequality incurred when using $\varphi$-divergence renormalization maps \eqref{eq:phi_map} by introducing a convergent sequence of approximations to the collision operator, parametrized by $N$ and denoted by $\mathcal C_N$:
\[
	\int_{\mathbb R^3} \varphi_N'\left(\frac{f(t,\bm x,\bm v)}{M(\bm v)}\right) \mathcal C(f) \, \text{d}\bm v
	=
	\underbrace{
		\int_{\mathbb R^3} \varphi_N'\left(\frac{f(t,\bm x,\bm v)}{M(\bm v)}\right) \mathcal C_N(f) \, \text{d}\bm v
	}_{\leq 0}
	+
	\underbrace{
		\int_{\mathbb R^3} \varphi_N'\left(\frac{f(t,\bm x,\bm v)}{M(\bm v)}\right) (\mathcal C(f) - \mathcal C_N(f)) \, \text{d}\bm v
	}_{\rightarrow 0 \text{ as } N\rightarrow\infty}.
\]
We emphasize that the index $N$ refers throughout to the polynomial order of the closure map $\beta_N$ and never to the collision operator itself. The operators $\mathcal C_N$ are auxiliary constructions for which the approximate $\varphi$-divergence entropy $\varphi_N$ is dissipated exactly. All numerical experiments reported in \cref{sec:numerical-results} are computed with the full physical binary collision operator $\mathcal C$ in \eqref{eq:collision}, not with $\mathcal C_N$.
Our formulation of $\mathcal C_N$ is based on approximations of the products of two distributions that appear in the integrand of $\mathcal C$ in \eqref{eq:collision}. Such approximate algebra and its use in constructing approximate collision operators based on approximate entropies have been considered before; see, e.g., \cite{borges2004,nivanen2003}. In contrast, our choice of approximate entropy functions \eqref{eq:phiNkullb} ensures that
\begin{itemize}
	\item the first member of the family $\mathcal C_N$, denoted by $\mathcal C_1$, as well as the linearization of any other member of $\mathcal C_N$ $(N>1)$ around the background Maxwellian $M(\bm v)$ coincide with the linearization of $\mathcal C$ around $M(\bm v)$, see \cref{thm:linearization-approximate-collision};
	\item our closure mappings \eqref{eq:phi_closure} remain polynomial, thereby enabling the use of optimal quadrature rules to carry out velocity integrals;
	\item the sequence of approximate collision operators $\mathcal C_N$ converges weakly to $\mathcal C$, see \cref{thm:convergence-approximate-collision}.
\end{itemize}
We conclude this section with a numerical example elucidating the properties of $\varphi$-divergences \eqref{eq:phiNkullb} for approximating the entropy production inequality \eqref{eq:ent_prod_coll}.

\begin{definition}
	For $N \in \mathbb N$ odd, the $N$-th order approximate product $P_N(z_1, z_2)$ of $z_1, z_2 \in \mathbb R$ is given by
	\begin{equation}
		P_N(z_1, z_2) = \big(z_1^{\frac{1}{N}} + z_2^{\frac{1}{N}} - 1 \big)^N,
	\end{equation}
	with $z^{\frac{1}{N}} = - |z|^{\frac{1}{N}}$ for $z < 0$.
\end{definition}
The divergence $\varphi_N$ and the approximate product are linked by the following
functional equation.
\begin{lemma}
	For $z_1, z_2 \in \mathbb R$,
	\begin{equation}\label{eq:functionalEquationPhi_N}
		\varphi_N'\big( P_N(z_1, z_2) \big) = \varphi_N'(z_1) + \varphi_N'(z_2).
	\end{equation}
\end{lemma}
The approximate product gives rise to a sequence of approximate collision operators.
\begin{definition}
	For $N \in \mathbb N$ odd and a fixed Maxwellian $M$, the approximate collision operator $\mathcal C_N$
	is defined for $f \in \mathcal D(\mathcal C)$ by
	\begin{equation}
		\mathcal C_N(f)(\bm v)
		= \int_{\mathbb R^3} \int_{\mathbb S^2} B(\bm v - \bm v_*, \bm \sigma) \background \background_*
		\left[
			P_N\left( \frac{\acute f}{\acute M}, \frac{\acute f_*}{\acute M_*} \right)
			-
			P_N\left( \frac{f}{M}, \frac{f_*}{M_*} \right)
			\right] \, \dd \bm \sigma \, \dd \bm v_*, \quad \bm v \in \mathbb R^3,
	\end{equation}
	where we use the notation $\acute f = f(\acute{\bm v})$, $\acute f_* = f(\acute{\bm v}_*)$, $f_* = f(\bm v_*)$
	to denote evaluation at the pre- and post-collisional velocities.
\end{definition}
Approximate collision operators share many properties with Boltzmann's collision operator,
viz., Galilean invariance relative to the background, a five-dimensional space of collision invariants comprising mass, momentum and energy, and a dissipation theorem.
The proofs of the following theorems are deferred to \labelcref{sec:details-approximate-collision}.

\begin{theorem}\label{thm:galilean-approximate-collision}
	Let $\bm u \in \mathbb R^3$ denote the mean velocity of the Maxwellian background $M$.
	The approximate collision operator $\mathcal C_N$ commutes with rotations relative to the moving background,
	\[
		\mathcal C_N \circ \mathsf T = \mathsf T \circ \mathcal C_N,
	\]
	where $\mathsf T = \tau_{- \bm u} \circ \mathsf T_O \circ \tau_{\bm u}$, $\tau_{\bm u}$ denotes the translation operator with shift $\bm u$ and $\mathsf T_O$ the rotation operator
	that rotates function arguments by $\bm O^{\top}$ for an orthogonal matrix $\bm O \in \mathbb R^{3 \times 3}$.
\end{theorem}
The next theorem shows that our approximate collision operators do not add any collision invariants
beyond the conservation of mass, momentum and energy.

\begin{theorem}\label{thm:conservation-approximate-collision}
	For a given test function $\psi$, the weak form of $\mathcal C_N$ vanishes for all $f \in \mathcal D(\mathcal C)$,
	\[
		\int_{\mathbb R^3} \psi \mathcal C_N(f) \, \dd \bm v = 0,
	\]
	if and only if $\psi$ is a collision invariant,
	\[
		\psi(\bm v) = a |\bm v|^2 + \bm b \cdot \bm v + c, \quad \bm v \in \mathbb R^3,
	\]
	with $a, c \in \mathbb R$ and $\bm b \in \mathbb R^3$.
\end{theorem}
Approximate collision operators dissipate a specific convex entropy.
\begin{theorem}\label{thm:dissipation-approximate-collision}
	The convex entropy density $\varphi_N: \mathbb R \to \mathbb R$,
	\[
		\varphi_N(z) = z N \Big( z^{\frac{1}{N}} \frac{N}{N + 1} - 1 \Big) + \frac{N}{N + 1},
	\]
	is dissipated by $\mathcal C_N$,
	\[
		\int_{\mathbb R^3} \varphi_N'(f / M) \mathcal C_N(f) \, \dd \bm v \leq 0,
	\]
	with equality if, and only if, $\varphi_N'(f / M)$ is a collision invariant.
\end{theorem}
\begin{corollary}\label{cor:kernel-approximate-collision}
	The approximate collision operator $\mathcal C_N$ vanishes for $f \in \mathcal D(\mathcal C)$ if, and only if,
	$f$ is an approximate Maxwellian,
	\[
		f(\bm v) = M(\bm v) \left( 1 + \frac{a + \bm b \cdot \bm v + c |\bm v|^2}{N} \right)^N,
		\quad \bm v \in \mathbb R^3,
	\]
	for $a, c \in \mathbb R$ and $\bm b \in \mathbb R^3$.
\end{corollary}
So far, we have shown that the sequence of approximate collision operators agrees qualitatively with the bilinear collision operator.
We now extend our study to the quantitative relation between $\mathcal C_N$ and $\mathcal C$.
First, we show that the linearization of any approximate collision operator around $M$ agrees with the linearization of $\mathcal C$ around $M$.

\begin{theorem}\label{thm:linearization-approximate-collision}
	The linearization of $\mathcal C_N$ around $M$ is independent of $N$ and agrees with the linearized collision operator $\mathcal L$ around $M$,
	\[
		\mathcal L g(\bm v) = \int_{\mathbb R^3} \int_{\mathbb S^2} B(\bm v - \bm v_*, \bm \sigma) \big(M(\acute{\bm v}) g(\acute{\bm v}_*) + M(\acute{\bm v}_*) g(\acute{\bm v}) - M(\bm v) g(\bm v_*) - M(\bm v_*) g(\bm v)\big)
		\, \dd \bm \sigma \, \dd \bm v_*, \quad \bm v \in \mathbb R^3.
	\]
\end{theorem}
In particular, since $\mathcal C_1$ is linear, we conclude that $\mathcal C_1 = \mathcal L$.
Our next main result demonstrates that the full sequence of operators recovers $\mathcal C$ in the limit.
\begin{theorem}\label{thm:convergence-approximate-collision}
	The sequence of approximate collision operators converges weakly to $\mathcal C$;
	namely for all $f \in \mathscr{D}(\mathcal C)$ with $f > 0$ almost everywhere
	and appropriate test functions $\psi: \mathbb R^3 \to \mathbb R$,
	\[
		\lim_{N \to \infty} \int_{\mathbb R^3} \psi \, \mathcal C_N(f) \, \dd \bm v = \int_{\mathbb R^3} \psi \, \mathcal C(f) \, \dd \bm v.
	\]
\end{theorem}
We conclude this section with a numerical study of the dissipation properties in the spatially homogeneous setting.
In this case, the Boltzmann equation simplifies to
\[
	\partial_t f = \mathcal C(f),
\]
subject to the initial condition $f_0 : \mathbb R^3 \to (0, \infty)$,
\[
	f(0, \cdot) = f_0.
\]
For the numerical solution we employ the Galerkin discretization of \cref{sec:MomSys}:
writing $f \approx \beta_N(\bm \lambda \cdot \bm m)$ and projecting the residual onto
$\operatorname{span} \bm m$ yields the ODE system \eqref{eq:moments-galerkin-pde} with
$\nabla_{\bm x} \equiv 0$, here with the collision term evaluated with the full operator $\mathcal C$,
\[
	\bm s(\bm \lambda) = \int_{\mathbb R^3} \mathcal C\big( \beta_N(\bm \lambda \cdot \bm m) \big) \bm m \, \dd \bm v.
\]
In the discrete formulation, we cannot guarantee the dissipation of Boltzmann's H functional, as, in general,
the logarithm of $\beta_N(\bm \lambda \cdot \bm m)$ cannot be expressed as a linear combination of the test
functions $\bm m$.
However, by construction of the approximate logarithm $\varphi_N'$, it holds
\[
	\varphi_N'\big( \beta_N(\bm \lambda \cdot \bm m) / M \big) = \bm \lambda \cdot \bm m.
\]
Thus, we can numerically study the dissipation properties of the approximate logarithm $\varphi_N'$
by forming
\[
	\int_{\mathbb R^3} \mathcal C\big( \beta_N(\bm \lambda \cdot \bm m) \big) \varphi_N'\big(
	\beta_N(\bm \lambda \cdot \bm m) / M
	\big) \, \dd \bm v
	= \int_{\mathbb R^3} \mathcal C\big( \beta_N(\bm \lambda \cdot \bm m) \big) \bm \lambda \cdot \bm m
	\, \dd \bm v
	= \bm s(\bm \lambda) \cdot \bm \lambda.
\]
The approximate collision operators $\mathcal C_N$ dissipate the approximate entropy $\varphi_N'$ and converge weakly to $\mathcal C$ by \cref{thm:convergence-approximate-collision}, while $\varphi_N'$ converges to the logarithmic function as $N \to \infty$ by \eqref{eq:convergence_phi_n}.
Combining these results, we expect the entropy rate $\dot\eta = \bm s(\bm \lambda) \cdot \bm \lambda$ to be nonpositive for $N$ large enough, with the magnitude of any entropy violation tending to $0$ as $N$ increases.

To demonstrate the control that the parameter $N$ provides over the violation of the $\varphi$-divergence production inequality \eqref{eq:phidiss}, we consider the spatially homogeneous setting of \eqref{eq:Boltzmann} and compute $\dot\eta$ for multiple values of $N$ with a fixed
polynomial space $\bm m$ containing all three-dimensional polynomials of at most total degree four.
We begin with a parametrized initial condition,
\begin{equation}\label{eq:homogeneous-initial}
	f_0(\bm v) = \frac{(3 / 5 \theta_\infty)^{-1}}{(6 / 5 \pi \theta_\infty)^{3/2}}
	v_1^2 \exp\left( - \frac{|\bm v|^2}{6 / 5 \theta_\infty} \right), \quad \bm v \in \mathbb R^3,
\end{equation}
where the scaling of the temperature ensures that the distribution function relaxes to a Maxwellian
of unit density, zero mean velocity and temperature $\theta_\infty$.
We compute the maximum entropy rate
\[
\dot\eta_\infty = \max_{t \in [0, \infty)} \dot\eta(t)
	\]
	for three choices of $N$, namely $N = 1, 3, 5$, over a range of the parameter $\theta_\infty$ in the initial condition; the corresponding entropy-rate evolutions are shown in
	\cref{fig:entropy_rate_matching,fig:entropy_small_deviation,fig:entropy_large_deviation}.
	For $\theta_\infty = 1$, that is, when the final distribution matches the choice of the background in $\beta_N$,
	the entropy rate stays non-positive throughout, as confirmed by
	\cref{fig:entropy_rate_matching}, where $\dot\eta$ begins with
	negative values and approaches $0$ exponentially fast.
	Owing to finite projection errors and solver tolerances, we do not reach $0$ exactly but values
	around $10^{-10}$.
	While $\dot\eta_\infty$ quickly diverges for $N = 1$ as $\theta_\infty$ increases,
	using a nonlinear closure substantially reduces the violation of the entropy inequality.
	Typically, we observe a reduction by more than one order of magnitude for the same $\theta_\infty$,
	as seen in \cref{fig:entropy_small_deviation,fig:entropy_large_deviation}.
	Going from $N = 3$ to $N = 5$ improves the results further, yielding a reduction by a factor
	of $4.5$ uniformly over the full range of $\theta_\infty$,
	compared to a reduction by a factor of $25$ for small $\theta_\infty$ and $30$ for large $\theta_\infty$
	between $N = 1$ and $N = 3$.
	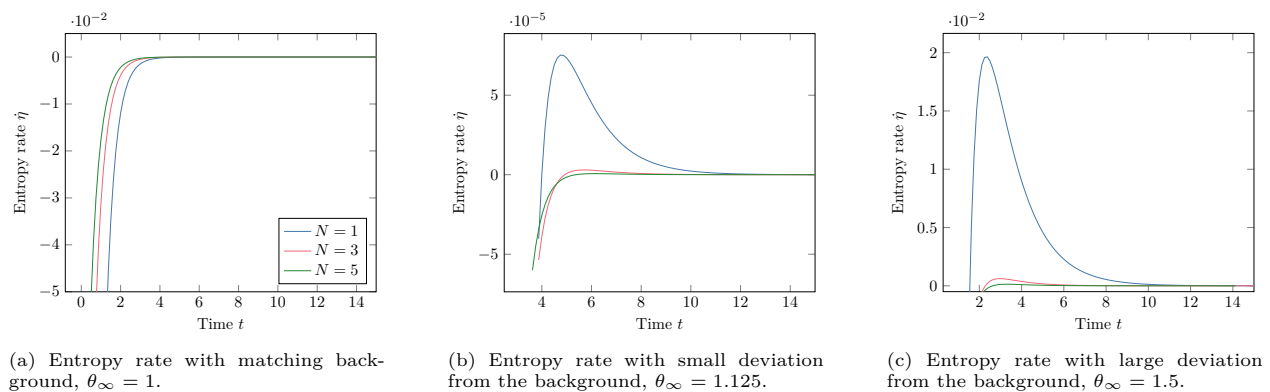
\begin{figure}[htp]
		\centering
		\begin{subfigure}[b]{0.3\textwidth}
			\begin{tikzpicture}[scale=.6]
				\begin{axis}
					[
						xlabel = {Time $t$},
						ylabel = {Entropy rate $\dot\eta$},
						legend pos = south east,
						ymin = -0.05,
						xmax = 15,
					]
					\addplot+[] table {data/entropy-n1-k4-1.0.dat};
					\addlegendentry{$N = 1$}
					\addplot+[] table {data/entropy-n3-k4-1.0.dat};
					\addlegendentry{$N = 3$}
					\addplot+[] table {data/entropy-n5-k4-1.0.dat};
					\addlegendentry{$N = 5$}
				\end{axis}
			\end{tikzpicture}
			\caption{Entropy rate with matching background, $\theta_\infty = 1$.}
			\label{fig:entropy_rate_matching}
		\end{subfigure}
		\hfill
		\begin{subfigure}[b]{0.3\textwidth}
			\begin{tikzpicture}[scale=.6]
				\begin{axis}
					[
						xlabel = {Time $t$},
						ylabel = {Entropy rate $\dot\eta$},
						legend pos = south east,
						restrict y to domain = -0.00007:0.0001,
						xmax = 15,
					]
					\addplot+[] table {data/entropy-n1-k4-1.125.dat};
					\addplot+[] table {data/entropy-n3-k4-1.125.dat};
					\addplot+[] table {data/entropy-n5-k4-1.125.dat};
				\end{axis}
			\end{tikzpicture}
			\caption{Entropy rate with small deviation from the background, $\theta_\infty = 1.125$.}
			\label{fig:entropy_small_deviation}
		\end{subfigure}
		\hfill
		\begin{subfigure}[b]{0.3\textwidth}
			\begin{tikzpicture}[scale=.6]
				\begin{axis}
					[
						xlabel = {Time $t$},
						ylabel = {Entropy rate $\dot\eta$},
						legend pos = south east,
						ymin = -0.0005,
						xmax = 15,
					]
					\addplot+[] table {data/entropy-n1-k4-1.5.dat};
					\addplot+[] table {data/entropy-n3-k4-1.5.dat};
					\addplot+[] table {data/entropy-n5-k4-1.5.dat};
				\end{axis}
			\end{tikzpicture}
			\caption{Entropy rate with large deviation from the background, $\theta_\infty = 1.5$.}
			\label{fig:entropy_large_deviation}
		\end{subfigure}
		\caption{Entropy rates for different values of $\theta_{\infty}$ in the initial condition \eqref{eq:homogeneous-initial}.}
	\end{figure}
	\section{Symmetric Dissipative Moment System Hierarchies}
	\label{sec:MomSys}
	In this section we demonstrate that in the finite-dimensional setting, the moment equations \eqref{eq:moments} with renormalization maps \eqref{eq:phi_map} are \textit{symmetric-dissipative hyperbolic} in the sense of \cite{kawashima2004}, and hence well-posed.

	To formulate the closed finite-dimensional moment equations, we consider Galerkin approximations of solutions of~\eqref{eq:moments} where $f = \beta_N(\bm \lambda \cdot \bm m)$
	for time- and space-dependent coefficients $\bm \lambda$ and a vector of polynomials $\bm m$
	constituting a basis of $\mathscr M$, a test space of dimension $m \in \mathbb N$ that includes the collision invariants;
	throughout, $K$ denotes the maximal polynomial degree of the basis $\bm m$.
	The Galerkin approximation of~\eqref{eq:moments} then reads
	\begin{subequations}\label{eq:moments-galerkin}
		\begin{equation}\label{eq:moments-galerkin-pde}
			\bm A_0(\bm \lambda) \partial_t \bm \lambda + \sum_{i = 1}^3 \bm A_i(\bm \lambda) \partial_{x_i} \bm \lambda = \bm s_N(\bm \lambda)
		\end{equation}
		in $(0, T) \times \mathbb R^3$, subject to the initial condition
		\begin{equation}
			\bm \lambda(0, \cdot) = \bm \lambda_0, \quad \text{in } \mathbb R^3,
		\end{equation}
	\end{subequations}
	where the matrices $\bm A_0, \dots, \bm A_3$ are given by
	\begin{gather}
		\bm A_0(\bm \lambda) = \int_{\mathbb R^3} \beta_N'(\bm \lambda \cdot \bm m) \bm m \, \bm m^{\top} \, \dd \bm v, \\
		\bm A_i(\bm \lambda) = \int_{\mathbb R^3} v_i \beta_N'(\bm \lambda \cdot \bm m) \bm m \, \bm m^{\top} \, \dd \bm v, \quad i=1,2,3,
	\end{gather}
	the source term $\bm s_N$ reads
	\begin{equation}
		\bm s_N(\bm \lambda) = \int_{\mathbb R^3} \bm m \mathcal C_N\big( \beta_N(\bm \lambda \cdot \bm m) \big) \, \dd \bm v,
	\end{equation}
	and $\bm \lambda_0$ satisfies
	\begin{equation}
		\int_{\mathbb R^3} \bm m \big( \beta_N(\bm \lambda_0 \cdot \bm m) - f_0 \big) \, \dd \bm v = 0, \label{eq:moment-initial}
	\end{equation}
	for a given initial distribution $f_0$ and all $\bm x \in \mathbb R^3$.
	Owing to our choice of $\beta_N$ with $N$ odd, the matrix $\bm A_0$ is symmetric and positive definite.
	Furthermore, the advection matrices $\bm A_1, \dots, \bm A_3$ are symmetric.
	Hence, the system (excluding the source) is symmetric hyperbolic, implying finite speed of propagation.
	We conclude this section by establishing sufficient conditions on the renormalization map $\beta_N$ for the well-posedness of~\eqref{eq:moments-galerkin}.
	In particular, we prove that the system~\eqref{eq:moments-galerkin-pde} is symmetric dissipative in the sense of Kawashima and Yong~\cite{kawashima2004},
	which implies that~\eqref{eq:moments-galerkin} is linearly well-posed.
	Moreover, under suitable conditions on the initial data, local-in-time existence of solutions can be established.
	To simplify the presentation we assume, without loss of generality, that the first five components of $\bm m$ form a basis of the space of collision invariants.
	\begin{definition}
		A system of $m \in \mathbb N$ first-order partial differential equations
		\begin{equation}
			\bm B_0(\bm u) \partial_t \bm u + \sum_{i=1}^3 \bm B_i(\bm u) \partial_{x_i} \bm u = \bm s(\bm u),
		\end{equation}
		posed for functions $u$ with open, convex codomain $\mathcal U \subseteq \mathbb R^m$ and
		\begin{equation}
			\mathcal N = \big\{ \bm \psi \in \mathbb R^m: \forall \bm u \in \mathcal U: \bm \psi^{\top} \bm s(\bm u) = 0 \big\}
		\end{equation}
		is called symmetric dissipative if
		\begin{enumerate}
			\item $\bm B_0(\bm u)$ is symmetric positive definite;
			\item $\bm B_i(\bm u)$, $i=1,2,3$, are symmetric;
			\item $\bm s(\bm u) = 0$ if and only if $\bm u \in \mathcal N$; and
			\item the linearization $\bm \nabla \bm s(\bm u)$ in $\bm u \in \mathcal N$ is symmetric and nonpositive definite, and its null space equals $\mathcal N$.
		\end{enumerate}
	\end{definition}
	\begin{theorem}
		The system~\eqref{eq:moments-galerkin-pde} is symmetric dissipative with
		\begin{equation}
			\mathcal N = \{ \bm \psi \in \mathbb R^m : \psi_k = 0, ~k > 5 \}
		\end{equation}
		and
		\begin{equation}
			\mathcal U = \{ \bm \lambda \in \mathbb R^m : \lambda_1 >  -N \}
		\end{equation}
		for $N > 1$ and $\mathcal U = \mathbb R^m$ for $N = 1$.
	\end{theorem}

	\begin{proof}
		It is evident that the matrices $\bm A_0(\bm \lambda), \dots, \bm A_3(\bm \lambda)$ are symmetric.
		To prove positive definiteness of $\bm A_0$, let $\bm \lambda \in \mathcal U$.
		Since $N$ is odd, $\beta_N'(g) = M \left( 1 + g/N \right)^{N - 1}$ is nonnegative, so for all $\bm b \in \mathbb R^m$,
		\begin{equation}
			\bm b^{\top} \bm A_0(\bm \lambda) \bm b
			= \int_{\mathbb R^3} \beta_N'(\bm \lambda \cdot \bm m)
			\big( \bm b \cdot \bm m \big)^2 \, \dd \bm v \geq 0,
		\end{equation}
		with equality if and only if the integrand vanishes;
		in that case the polynomial $\bm b \cdot \bm m$ vanishes on the open set
		$P = \{ \bm v \in \mathbb R^3 : \beta_N'(\bm \lambda \cdot \bm m)(\bm v) > 0 \}$,
		which is nonempty by assumption on $\bm \lambda$,
		so $\bm b \cdot \bm m$ vanishes identically on $\mathbb R^3$ and hence $\bm b = 0$.
		Next, we show that $\mathcal N$ is given as asserted.
		To that end, let $\bm \psi \in \mathbb R^m$ with
		\begin{equation}
			\bm \psi^{\top} \bm s_N(\bm \lambda) = 0
		\end{equation}
		for all $\bm \lambda \in \mathcal U$.
		Without loss of generality, we can assume that $\bm \psi \in \mathcal U$,
		since we can replace $\bm \psi$ by $\tilde{\bm \psi} = \bm \psi + c \bm e_1$
		for $c > 0$ sufficiently large.
		In particular, we have
		\[
			0 = \bm \psi^{\top} \bm s_N(\bm \psi)
			= \int_{\mathbb R^3} \varphi_N'\left( \frac{\beta_N(\bm \psi \cdot \bm m)}{M} \right) \mathcal
			C_N\big( \bm \psi \cdot \bm m \big) \, \dd \bm v,
		\]
		which, by \cref{thm:dissipation-approximate-collision},
		implies that $\bm \psi \cdot \bm m$ is a collision invariant, that is, $\psi_k = 0$ for $k > 5$.
		From \cref{cor:kernel-approximate-collision} it follows that
		$\bm s_N(\bm \lambda) = 0$ if and only if $\bm \lambda \cdot \bm m$
		is a collision invariant and hence $\bm \lambda \in \mathcal N$.
		Lastly, we compute the linearization of $\bm s_N$.
		Differentiating the weak form of $\bm s_N$ and, for $\bm \lambda \in \mathcal N$, symmetrizing over pre- and post-collisional velocities as in \eqref{eq:symm_coll}, under which the integration weight is invariant, gives
		\begin{multline}
			\bm \nabla \bm s_N(\bm \lambda) = -\frac{1}{4}
			\int_{\mathbb R^3} \int_{\mathbb R^3} \int_{\mathbb S^2}
			B(\bm v - \bm v_*, \bm \sigma) M(\bm v) M(\bm v_*)
			\left( 1 + \frac{\bm \lambda \cdot \bm m(\bm v) + \bm \lambda \cdot \bm m(\bm v_*)}{N} \right)^{N-1} \\
			\times \big( \bm m(\acute{\bm v}) + \bm m(\acute{\bm v}_*) - \bm m(\bm v) - \bm m(\bm v_*)
			\big) \big(
			\bm m(\acute{\bm v}) + \bm m(\acute{\bm v}_*) - \bm m(\bm v) - \bm m(\bm v_*)
			\big)^{\top}
			\, \dd \bm \sigma \, \dd \bm v_* \, \dd \bm v.
		\end{multline}
		This representation shows that $\bm \nabla \bm s_N(\bm \lambda)$ is symmetric and nonpositive definite.
		Furthermore, we infer that a vector $\bm b$ lies in its kernel if and only if $\bm b \cdot \bm m$
		is a collision invariant, i.e., $\bm b \in \mathcal N$.
	\end{proof}

	\section{The Discontinuous Galerkin Finite Element Approximation}
	\label{sec:DGFEM}

	In this section we approximate solutions of \eqref{eq:moments-galerkin} using
	a discontinuous Galerkin approximation in the position and time variables,
	and highlight some of the properties that the discontinuous Galerkin finite element (DGFE)
	moment solutions inherit from the initial--boundary value problem of the Boltzmann equation.
	Moreover, we show that the resulting space-time DGFE moment approximation is entropy stable,
	in the sense that the approximate entropy of solutions of \eqref{eq:DGform} is bounded from above
	by that of the initial data.

	For the discretization of \eqref{eq:moments-galerkin} in space-time $(0, T) \times \Omega$,
	we consider the discontinuous Galerkin finite element method~\cite{dipietro:2012kx}.
	To introduce the DGFE approximation space, let $\mathcal{P}_p(\kappa \times I^n)$
	denote the set of $D$-variate polynomials of degree at most $p$ in an element domain $\kappa\subset\mathbb{R}^D$,
	and let $I^n = ]t^n,t^{n+1}[$ denote the $n$-th time interval,
	with $n$ a nonnegative integer.
	For a sequence of quasi-uniform triangular meshes $(\mathcal T_h)_h$,
	we indicate by $V^{h,p}((0,T)\times\Omega)$ the DGFE approximation space,
	\begin{equation}\label{eq:DGspace}
		V^{h,p}((0,T)\times\Omega)=\{g\in{}L^2((0,T)\times\Omega):\ g|_{\kappa\times I^n}\in\mathcal{P}_p(\kappa \times I^n ), \ \forall\kappa\in\mathcal{T}^h\},
	\end{equation}
	and by $V^{h,p}((0,T)\times\Omega,\mathscr{M})$ the extension of $V^{h,p}((0,T)\times\Omega)$ to $\mathscr{M}$-valued functions.

	To facilitate the presentation of the DGFE formulation, we introduce some further notational conventions.
	For any mesh $\mathcal T_h$ in the sequence, we indicate by
	\begin{equation}
		\mathcal{I}^h=\{\mathrm{int}(\partial\kappa\cap\partial\hat{\kappa}):\kappa,\hat{\kappa}\in\mathcal{T}^h,\kappa\neq\hat{\kappa}\}
	\end{equation}
	the collection of inter-element edges, by
	\begin{equation}
		\mathcal{G}^h=\{\mathrm{int}(\partial\kappa\cap\partial\Omega):\kappa\in\mathcal{T}^h\}
	\end{equation}
	the collection of boundary edges and by $\mathcal{S}^h=\mathcal{G}^h\cup\mathcal{I}^h$ their union.
	With every edge we associate a unit normal vector~$\bm{\nu}^e$.
	The orientation of~$\bm{\nu}^e$ is arbitrary except on boundary edges where $\bm{\nu}^e=\bm{n}|_e$ is the outward normal.
	For all interior edges, let $\kappa_{\pm}^e\in\mathcal{T}^h$ be the two elements adjacent to the edge~$e$
	such that the orientation of $\bm{\nu}^e$ is exterior to~$\kappa_{+}$.
	We define the edge-wise vector-valued jump and scalar mean operators for an interior interface $e\in\mathcal I$ as
	\begin{equation}
		\label{eq:jumpmean}
		\jump{c}=
		c_+ \bm\nu^{\kappa}_+ + c_- \bm\nu^{\kappa}_-;
		\quad
		\mean{c}=
		\frac{c_++c_-}{2},
	\end{equation}
	where the subscripts $(\cdot)_+$ and $(\cdot)_-$ refer to the restriction of the traces of $(\cdot)|_{\kappa_+}$ and $(\cdot)|_{\kappa_-}$ to~$e$.

	To derive the DGFE formulation of the closed moment system \eqref{eq:moments-galerkin},
	we consider space-time test functions $\{w_i\}_{i=0}^n$, $w_i\in V^{h,p}((0,T)\times\Omega)$,
	which constitute the vector $\bm w(t,\bm x)$. We note that for any $w_i\in V^{h,p}((0,T)\times\Omega)$
	and $m_i\in\mathscr M$ we have $(\bm w \cdot \bm m) \in V^{h,p}((0,T) \times \Omega,\mathscr M)$,
	and there holds
	\begin{equation}
		\label{eq:DGweight}
		\sum_n\sum_{\kappa\in\mathcal{T}^h}\int_{I^n \times \kappa \times \mathbb R^3}
		\bm w \cdot \bm m\,
		\big(
		\partial_{t}
		\beta_N(\bm \lambda \cdot \bm m)
		+
		\bm v\cdot \partial_{\bm x}
		\beta_N(\bm \lambda \cdot \bm m)
		-
		\,
		\mathcal{C}_N(\beta_N(\bm \lambda \cdot \bm m))
		\big)
		\, \dd\bm v \, \dd\bm x \, \dd t
		=
		0.
	\end{equation}
	Using the product rule and integration by parts, \eqref{eq:DGweight} can be reformulated in weak form,
	see~\cite{barth2006,abdelmalik2016a,abdelmalik2022}.
	In this work, we depart from these formulations in the treatment of the interface flux stabilization by replacing it with a general
	dissipation-flux operator $\bm D: \mathbb R^m \times \mathbb R^m \times\mathbb R^d \to \mathbb R^m \times \mathbb R^d$
	that complies with
	\begin{enumerate}
		\item \label{cond:DGdim} dimensional consistency, i.e., $\bm D(\bm\lambda_+,\bm\lambda_-,\bm\nu^{\kappa})$
		      has the same physical unit as $\int_{\mathbb R^3}  (\bm m \otimes \bm  v) \beta_N(\bm\lambda\cdot\bm m)\,\dd\bm v$;
		\item \label{cond:DGorth} and the orthogonality condition $\jump{\bm w}:\bm D(\bm\lambda,\bm\lambda,\bm \nu) =0$.
	\end{enumerate}
	Assuming that solutions of \eqref{eq:moments-galerkin} are continuous across elements,
	we may replace the interface integrands that appear when performing integration by parts,
	\begin{equation}\label{eq:consistmod}
		\int\limits_{\mathrlap{\mathbb R^3}}
		\bm v \cdot \jump{\bm w \cdot \bm m } \, \mean{\beta_N(\bm\lambda\cdot\bm m)}
		\, \dd\bm v
		=
		\int\limits_{\mathrlap{\mathbb R^3 \times [0,1]}}
		\bm v \cdot \jump{\bm w \cdot \bm m } \,
		\beta_N(\hat{\bm \lambda}(\zeta) \cdot \bm m)
		\, \dd\zeta \, \dd\bm v
		+
		\jump{\bm w}:\bm D(\bm \lambda_+,\bm \lambda_-,\bm \nu^{\kappa}).
	\end{equation}
	These conditions ensure that the equality in \eqref{eq:consistmod} holds
	when $\bm\lambda^+ = \bm\lambda^-$, i.e., any solution to \eqref{eq:moments-galerkin}
	that is sufficiently regular in the aforementioned sense also satisfies
	\begin{subequations}
		\begin{align}
			\label{eq:DGform}
			a_t(\bm \lambda;\bm w) + a_x^{\Omega}(\bm \lambda;\bm w) +  a_x^{\partial\Omega}(\bm \lambda;\bm w)
			=
			s(\bm \lambda; \bm w) \quad
			\forall w_i \in V^{h,p}((0,T)\times\Omega),
		\end{align}
		\text{where}
		\begin{align}
			a_t(\bm \lambda;\bm w)
			                                       & =
			\begin{multlined}[t]
				\sum_n\sum_{\kappa \in \mathcal{T}^h}
				\int\limits_{\mathrlap{\kappa \times \mathbb R^3}}
				\left(\bm w(t_-^{n+1}) \cdot \bm m\right) \,
				\beta_N(\bm \lambda(t_-^{n+1})\cdot\bm m)
				-
				\left(\bm w(t_+^{n}) \cdot \bm m\right) \,
				\beta_N(\bm \lambda(t_-^{n}) \cdot \bm m)
				\, \dd\bm v \, \dd\bm x
				\\
				-
				\sum_n\sum_{\kappa \in \mathcal{T}^h}
				\int\limits_{\mathrlap{I^n \times \kappa \times \mathbb R^3}}
				\beta_N(\bm \lambda \cdot \bm m)\,\partial_t \bm w \cdot \bm m
				\, \dd\bm v \, \dd\bm x \, \dd t
				\taghere
			\end{multlined} \label{eq:atDG}
			\\
			a_x^{\Omega}(\bm\lambda;\bm w)
			                                       & =
			\begin{multlined}[t]
				\sum_n
				\sum_{e\in\mathcal{I}^h}
				\int\limits_{\mathrlap{I^n \times e \times \mathbb R^3 \times [0,1]}}
				\bm v \cdot \jump{\bm w \cdot \bm m } \,
				\beta_N(\hat{\bm \lambda}(\zeta) \cdot \bm m)
				\, \dd\zeta \, \dd\bm v
				+
				\jump{\bm w}:\bm D(\bm \lambda_+,\bm \lambda_-,\bm \nu^{\kappa})
				\, \dd\bm s \, \dd t
				\\
				-
				\sum_n
				\sum_{\kappa\in\mathcal{T}^h}
				\int\limits_{\mathrlap{ I^n\times \kappa\times \mathbb R^3}}
				\beta_N(\bm \lambda \cdot \bm m)
				\,
				\bm v \cdot\partial_{\bm x}\bm w \cdot \bm m
				\, \dd\bm v \, \dd\bm x \, \dd t
				\taghere
			\end{multlined} \label{eq:aOmegaDG}
			\\
			a_x^{\partial\Omega}(\bm\lambda;\bm w) & =
			\sum_n
			\sum_{e\in\mathcal{G}^h}
			\int\limits_{\mathrlap{I^n \times e \times \mathbb R^3_{\text{out}}}}
			\bm w \cdot \bm m\, (\bm v \cdot \bm \nu^{\kappa})\,\beta_N(\bm\lambda\cdot\bm m)
			\, \dd\bm v \, \dd\bm s \, \dd t
			+
			\sum_n
			\sum_{e\in\mathcal{G}^h}
			\int\limits_{\mathrlap{I^n \times e \times \mathbb R^3_{\text{in}}}}
			\bm w \cdot \bm m\, (\bm v \cdot \bm \nu^{\kappa})\, f_w
			\, \dd\bm v \, \dd\bm s \, \dd t
			\label{eq:abndDG}
			\\
			s(\bm \lambda;\bm w)
			                                       & =
			\sum_n\sum_{\kappa\in\mathcal{T}^h}
			\int\limits_{\mathrlap{I^n \times \kappa \times \mathbb R^3}}
			\bm w \cdot \bm m
			\,
			\mathcal{C}_N(\beta_N(\bm \lambda \cdot \bm m))
			\, \dd\bm v \, \dd\bm x \, \dd t,
			\label{eq:CollDG}
		\end{align}
	\end{subequations}
	where we weakly imposed boundary data on the boundary edges in \eqref{eq:abndDG}.
	Finally, we arrive at the full DGFE moment discretization of \eqref{eq:moments-galerkin}
	where we replace each component of $\bm \lambda$ in \eqref{eq:DGform} by an approximation $\lambda_i^{h,p}$ in $V^{h,p}((0,T)\times\Omega)$ according to
	\begin{multline}
		\label{eq:DGFEform}
		\text{\em Find } \lambda_i^{h,p}\in{}V^{h,p}((0,T)\times\Omega):
		\\
		a_t(\bm\lambda^{h,p};\bm w)
		+
		a_x^{\Omega}(\bm\lambda^{h,p};\bm w)
		+
		a_x^{\partial\Omega}(\bm\lambda^{h,p};\bm w)
		=
		s(\bm\lambda^{h,p};\bm w)
		\\
		\forall w_i \in{} V^{h,p}((0,T)\times\Omega).
	\end{multline}
	As a first consequence of the discrete weak form, we establish conservation in phase space.
	\begin{lemma}
		For any constant space-time test function $\tilde{\bm w}$ such that
		$\tilde{\bm w}\cdot\bm m$ is a collision invariant conserved by the boundary conditions,
		solutions of \eqref{eq:DGFEform} discretely conserve $\tilde{\bm w} \cdot \bm m$ over time,
		\begin{equation}\label{eq:disctconserv}
			\int\limits_{\mathrlap{\Omega \times \mathbb R^3}}
			(\tilde{\bm w} \cdot \bm m) \,
			\beta_N(\bm \lambda^{h,p}(t_-^{n}, \cdot)\cdot\bm m)
			\, \dd\bm v \, \dd\bm x
			=
			\int\limits_{\mathrlap{\Omega \times \mathbb R^3}}
			(\tilde{\bm w} \cdot \bm m) \,
			\beta_N(\bm \lambda^{h,p}(0, \cdot) \cdot \bm m)
			\, \dd\bm v \, \dd\bm x,
		\end{equation}
		for $n \in \mathbb N$.
	\end{lemma}
	\begin{proof}
		The source term vanishes, $s(\bm \lambda; \tilde{\bm w}) = 0$, by the collision invariance in
		\cref{thm:conservation-approximate-collision}.
		In the spatial form $a_x^\Omega$ the volumetric terms vanish because $\tilde{\bm w}$ is constant
		in space and the interface terms vanish because $\tilde{\bm w}$ is continuous on $\Omega$, so
		$a_x^\Omega(\bm \lambda; \tilde{\bm w}) = 0$; the boundary form vanishes,
		$a_x^{\partial \Omega}(\bm \lambda; \tilde{\bm w}) = 0$, by the hypothesis that
		$\tilde{\bm w} \cdot \bm m$ is conserved by the boundary conditions.
		Hence $a_t(\bm \lambda; \tilde{\bm w}) = 0$, and since $\tilde{\bm w}$ is constant in time this
		reduces, per time slab, to
		\begin{equation}
			\int\limits_{\mathrlap{\Omega \times \mathbb R^3}}
			(\tilde{\bm w} \cdot \bm m) \,
			\beta_N(\bm \lambda^{h,p}(t_-^{n+1}, \cdot)\cdot\bm m)
			\, \dd\bm v \, \dd\bm x
			=
			\int\limits_{\mathrlap{\Omega \times \mathbb R^3}}
			(\tilde{\bm w} \cdot \bm m) \,
			\beta_N(\bm \lambda^{h,p}(t_-^{n}, \cdot) \cdot \bm m)
			\, \dd\bm v \, \dd\bm x.
		\end{equation}
		Telescoping in $n$ gives the claim.
	\end{proof}

	Having established the conservation properties, we turn to the stability of solutions of the discrete form,
	in particular entropy stability.
	From the definition of $\beta_N$ and the approximate entropy $\varphi_N$, we know
	\[
		\varphi_N' \big( \beta_N(\bm \lambda \cdot \bm m) / M \big) = \bm \lambda \cdot \bm m
	\]
	and hence for the discrete solution $f_h = \beta_N(\bm \lambda \cdot \bm m)$ we have that
$\varphi_N'(f_h / M)$ is an element of our test space.
	This observation is the main idea of the proof, presented in \labelcref{sec:details-dgfem},
	of the following entropy-stability theorem.
	\begin{theorem}
		\label{them:entropy-stability}
		If the dissipation-flux operator satisfies the positivity condition
		\begin{equation}\label{eq:D>0}
			\jump{\bm z}:\bm D(\bm z_+,\bm z_-,\bm \nu) \geq 0 \quad \forall\bm z \in \mathbb R^m,
		\end{equation}
		and the background $\background$ conforms to the imposed boundary conditions,
		then the approximate entropy is monotonically decreasing.
		In particular, it is bounded from above by the initial approximate entropy,
		\begin{equation}\label{eq:entstab}
			\int_{\Omega\times\mathbb R^3} \background \varphi_N\left( \frac{\beta_N(\bm \lambda^{h, p}(t_-^n, \cdot) \cdot \bm m)}{\background} \right) \, \dd\bm v \, \dd \bm x
			\leq
			\int_{\Omega\times\mathbb R^3} \background \varphi_N\left( \frac{\beta_N(\bm\lambda^{h, p}(0, \cdot)\cdot\bm m)}{\background} \right) \, \dd\bm v \, \dd \bm x
		\end{equation}
		for all discrete times $n \in \mathbb N$.
	\end{theorem}
	We stress that this exact entropy-stability guarantee pertains to the approximate operator $\mathcal C_N$, for which $\varphi_N$ is an exact entropy; the computations reported below instead employ the full physical operator $\mathcal C$, a controlled variational crime discussed and quantified in \cref{sec:approx_coll}.

	\section{Numerical Results}
	\label{sec:numerical-results}
	All steady states reported in this section are obtained by pseudo-transient continuation.
	Exploiting the entropy stability of the fully implicit space-time discretization (\cref{them:entropy-stability}),
	we advance the solution with a time-step size of about $10^3$,
	which drives the iteration to a steady state in roughly six steps for all problems considered.
	\subsection{Supersonic Nozzle Flow}
	Supersonic nozzle flows at low ambient pressure are central to a range of technological
	applications operating in the transitional flow regime: satellite thrusters for
	stationkeeping and orbit maintenance, in particular in very low Earth orbit~\cite{andreussi2022},
	purge-gas flushing systems that protect the collector optics in extreme ultraviolet
	lithography sources~\cite{he2026}, and neutral gas exhaust and fueling systems
	in fusion reactors~\cite{li2026} all rely on an accurate description of the gas
	expansion through a nozzle into a low-pressure environment.
	The supersonic expansion of a rarefied gas through a nozzle has been studied
	experimentally since the pioneering electron-beam density measurements
	of Rothe~\cite{rothe1971}.
	As our quantities of interest, we consider, analogous to these classical
	studies, the bulk flow fields of temperature and velocity throughout the
	nozzle, the Mach number statistics in the plume, the local Knudsen number
	along the nozzle axis, and, along the centerline, the temperature,
	velocity magnitude and heat flux profiles.

	We consider the steady expansion of argon through a planar
	converging-diverging nozzle, with the geometry taken
	from~\cite{krafft2025}.
	At the inflow, we prescribe a fixed pressure of \SI{700}{\pascal}, while the
	outflow pressure is set to \SI{7}{\pascal}.
	The nozzle walls are kept at a constant temperature of \SI{300}{\kelvin}
	and modeled with diffusely reflecting boundary conditions.
	Exploiting the symmetry of the geometry, we impose a symmetry condition
	along the centerline and restrict the computation to the upper half of the
	nozzle.
	As a reference solution, we compute a DSMC simulation of the same
	configuration with the open-source solver SPARTA\footnote{\url{http://sparta.github.io}}~\cite{plimpton2019}.

	\Cref{fig:nozzle-argon} shows the steady-state moment solution in terms of
	the temperature and velocity magnitude fields.
	As the gas expands through the diverging section of the nozzle, it
	accelerates from subsonic conditions near the throat into a supersonic
	plume, while its temperature drops correspondingly.
	\Cref{tab:nozzle-mach} reports the resulting Mach number statistics in the
	plume ($x > \SI{13.84}{\milli\meter}$): the maximum and the 99th-percentile
	Mach number, the mean Mach number in the plume core, and the value at the
	nozzle exit plane.
	The moment solution underpredicts the DSMC Mach numbers throughout, but the
	gap closes markedly under spatial refinement from the coarse to the
	refined mesh, mirroring the convergence behavior discussed below for the
	centerline profiles.

	We quantify the local degree of rarefaction using the mean free path
$\ell$ of Bird's variable hard sphere model~\cite{bird1994},
	\[
		\ell = \frac{1}{\sqrt{2} \, \pi d_\text{ref}^2 n}
		\left( \frac{\theta_\text{ref}}{\theta} \right)^{\omega - 1/2},
	\]
	with the local number density $n$ and local temperature $\theta$.
	The argon specific parameters are $d_\text{ref} = \SI{4.17}{\angstrom}$,
$\theta_\text{ref} = \SI{273}{\kelvin}$ and $\omega = 0.81$.
	The associated Knudsen number $\kn = \ell / r_\text{throat}$ is based
	on the throat radius $r_\text{throat} = \SI{0.441}{\milli\meter}$.
	As reported in \cref{tab:nozzle-mfp}, the Knudsen number remains small in
	the inlet and converging section ($\kn \approx 0.023$, near-continuum),
	grows into the transitional regime across the diverging section and the
	nozzle lip ($\kn$ up to $1.0$--$1.3$), and reaches the rarefied regime in
	the plume, peaking at $\kn \approx 2.0$--$2.1$ (locally as high as $6.2$)
	just downstream of the lip before relaxing to $\kn \approx 1.5$--$1.7$
	further downstream.

	\begin{table}
		\centering
		\begin{subtable}[t]{\textwidth}
			\centering
			\caption{Mach number statistics in the nozzle plume
				($x > \SI{13.84}{\milli\meter}$): maximum Mach number $\text{Ma}_\text{max}$,
				99th-percentile Mach number $\text{Ma}_{p99}$, mean Mach number in the plume core
				$\overline{\text{Ma}}_\text{core}$, and Mach number at the nozzle exit plane
				$\text{Ma}_\text{exit}$.}
			\label{tab:nozzle-mach}
			\begin{tabular}{@{}l S S S S@{}}
				\toprule
				{}                & {$\text{Ma}_\text{max}$} & {$\text{Ma}_{p99}$} & {$\overline{\text{Ma}}_\text{core}$} & {$\text{Ma}_\text{exit}$} \\
				\midrule
				DSMC              & 3.66                     & 3.30                & 2.40                                 & 2.65                      \\
				Moments (refined) & 3.07                     & 2.86                & 2.07                                 & 2.26                      \\
				Moments (coarse)  & 2.71                     & 2.56                & 1.82                                 & 1.99                      \\
				\bottomrule
			\end{tabular}
		\end{subtable}

		\vspace{1.5ex}
		\begin{subtable}[t]{\textwidth}
			\centering
			\caption{Mean free path $\ell$ and Knudsen number $\kn = \ell / r_\text{throat}$,
				evaluated as the median over four regions along the nozzle axis; the throat is
				located at $x = \SI{3.84}{\milli\meter}$.}
			\label{tab:nozzle-mfp}
			\small
			\begin{tabular}{@{}l S S S l@{}}
				\toprule
				{Region}                                          & {$\ell_\text{DSMC}$}                       & {$\ell_\text{coarse}$} & {$\ell_\text{refined}$} & {Knudsen number regime}                        \\
				{}                                                & \multicolumn{3}{c}{in \unit{\micro\meter}}                                                                                                     \\
				\midrule
				Inlet, converging ($x < \SI{3.84}{\milli\meter}$) & 9.8                                        & 10.2                   & 10.2                    & small, $\kn \approx 0.023$                     \\
				Diverging (\SIrange{3.84}{13.84}{\milli\meter})   & 59                                         & 95                     & 94                      & transitional, $\kn$ up to $1.0$--$1.3$         \\
				Near plume (\SIrange{13.84}{30}{\milli\meter})    & 894                                        & 928                    & 920                     & rarefied, $\kn \approx 2.0$--$2.1$ (max $6.2$) \\
				Far plume (\SIrange{30}{113.75}{\milli\meter})    & 743                                        & 666                    & 671                     & rarefied, $\kn \approx 1.5$--$1.7$             \\
				\bottomrule
			\end{tabular}
		\end{subtable}
		\caption{Plume statistics for the DSMC reference and the moment solution on the coarse and
			refined meshes of \cref{fig:nozzle-argon-centerline}.}
	\end{table}

	For this strongly nonequilibrium expansion, the nonlinear renormalization
	introduced in \cref{sec:approx_coll} is essential for a well-posed and
	physically admissible moment system. With the linear closure $N = 1$, the
	pseudo-transient continuation reports convergence, yet the resulting steady
	state is not realizable: a localized pocket of a few nonphysical cells
	with negative densities, pressures and temperatures.
	In contrast, the nonlinear closure with $N = 3$ converges robustly to
	a physically admissible steady state.

	\Cref{fig:nozzle-argon-centerline} compares the centerline profiles of
	temperature, velocity magnitude and heat flux obtained with the moment
	method against the DSMC reference.
	Upon spatial refinement, the moment solution approaches the DSMC data from
	below, both for the temperature in
	\cref{fig:nozzle-argon-centerline-temperature} and for the velocity
	magnitude in \cref{fig:nozzle-argon-centerline-velocity}, indicating that
	the spatial discretization error dominates the overall error budget for
	this problem.
	In contrast, refining the moment discretization beyond $N = 3, K = 6$ does not
	have a significant effect on the solution, so the moment truncation error
	is subdominant compared to the discretization error of the DGFE method.
	Finally, as shown in \cref{fig:nozzle-argon-centerline-heat-flux}, the
	DSMC heat flux exhibits pronounced stochastic fluctuations.
	The reference solution was obtained by assuming steady state once the
	mass flow rates at the inflow and outflow had balanced, followed by
	time-averaging of the sampled particle data; this procedure proved
	insufficient to damp the statistical noise inherent to the particle
	method, in clear contrast to the noise-free moment solution.
	\begin{figure}
		\centering
		\includegraphics[width=\textwidth]{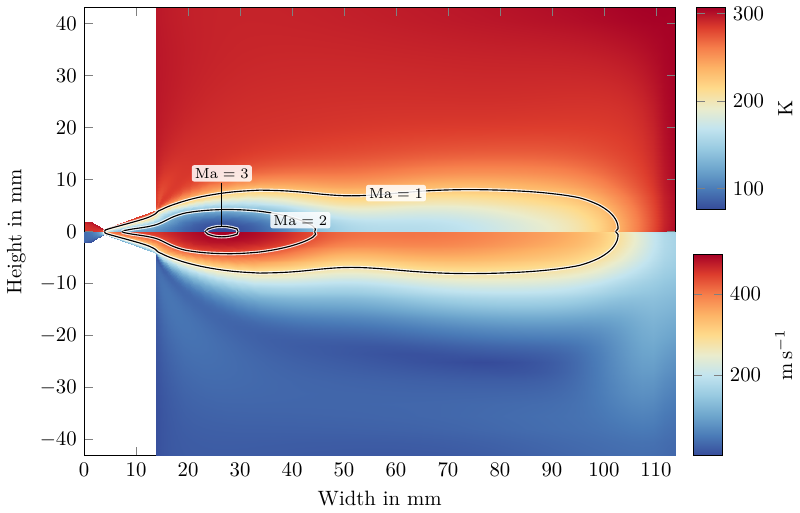}
		\caption{Steady-state solution of the supersonic nozzle flow of argon. The upper half shows the temperature field and the lower half the magnitude of the velocity field, mirrored at the symmetry axis to recover the full nozzle geometry.
			Both halves carry the same isolines of the local Mach number
			$\text{Ma} = |\bm u| / \sqrt{\gamma R_s \theta}$, with the adiabatic index
			$\gamma = 5/3$ of a monatomic gas.
			The maximum $\text{Ma}_\text{max} = 3.07$ of \cref{tab:nozzle-mach} is attained on the
			centerline at $x = \SI{26.5}{\milli\meter}$, where the temperature drops to \SI{76}{\kelvin}.
			Nonlinear renormalization $N=3$ with polynomial degree $K=6$ and piecewise constant DG trial functions on a triangular mesh with \num{e4} elements.
			The fields shown are a linear reconstruction of the piecewise constant discrete solution, obtained by interpolating the cell values across the mesh for visualization.}
		\label{fig:nozzle-argon}
	\end{figure}

	\def\momentfilerefined{data/centerline_moment_dedup.csv}
	\def\momentfilecoarse{data/centerline_moment_nx0_dedup.csv}
	\def\dsmcfile{data/centerline_dsmc_dedup.csv}
	\begin{figure}
		\centering
		\pgfplotslegendfromname{nozzle-centerline-legend}

		\vspace{1ex}
		\begin{subfigure}[b]{0.48\textwidth}
			\centering
			\begin{tikzpicture}
				\begin{axis}[
						xlabel={$x$ in \unit{\milli\meter}},
						ylabel={$\theta$ in \unit{\kelvin}},
						width=7.5cm,
						height=4.3cm,
						xmin=0, xmax=113.75,
						legend to name=nozzle-centerline-legend,
						legend columns=-1,
						legend cell align=left,
						legend style={draw=none},
						every axis plot/.append style={line width=0.9pt},
					]
					\addplot+[no marks, color1] table[x expr=\thisrow{X}*1000, y=temperature, col sep=comma] {\momentfilerefined};
					\addlegendentry{Moments (refined)}
					\addplot+[no marks, color1, dashed] table[x expr=\thisrow{X}*1000, y=temperature, col sep=comma] {\momentfilecoarse};
					\addlegendentry{Moments (coarse)}
					\addplot+[no marks, color2] table[x expr=\thisrow{X}*1000, y=temperature_total, col sep=comma] {\dsmcfile};
					\addlegendentry{DSMC}
				\end{axis}
			\end{tikzpicture}
			\caption{Temperature.}
			\label{fig:nozzle-argon-centerline-temperature}
		\end{subfigure}
		\hfill
		\begin{subfigure}[b]{0.48\textwidth}
			\centering
			\begin{tikzpicture}
				\begin{axis}[
						xlabel={$x$ in \unit{\milli\meter}},
						ylabel={$|\bm u|$ in \unit{\meter\per\second}},
						width=7.5cm,
						height=4.3cm,
						xmin=0, xmax=113.75,
						every axis plot/.append style={line width=0.9pt},
					]
					\addplot+[no marks, color1] table[x expr=\thisrow{X}*1000, y expr=sqrt((\thisrow{velocity_0})^2+(\thisrow{velocity_1})^2+(\thisrow{velocity_2})^2), col sep=comma] {\momentfilerefined};
					\addplot+[no marks, color1, dashed] table[x expr=\thisrow{X}*1000, y expr=sqrt((\thisrow{velocity_0})^2+(\thisrow{velocity_1})^2+(\thisrow{velocity_2})^2), col sep=comma] {\momentfilecoarse};
					\addplot+[no marks, color2] table[x expr=\thisrow{X}*1000, y expr=sqrt((\thisrow{velocity_0})^2+(\thisrow{velocity_1})^2+(\thisrow{velocity_2})^2), col sep=comma] {\dsmcfile};
				\end{axis}
			\end{tikzpicture}
			\caption{Velocity magnitude.}
			\label{fig:nozzle-argon-centerline-velocity}
		\end{subfigure}

		\vspace{1em}
		\begin{subfigure}[b]{\textwidth}
			\noindent\makebox[\textwidth][r]{%
				\begin{tikzpicture}
					\begin{groupplot}[
							group style={group size=2 by 1, horizontal sep=0.9cm},
							width=7.3cm,
							height=4.3cm,
							xlabel={$x$ in \unit{\milli\meter}},
							ylabel={$q_x$ in \unit{\watt\per\meter\squared}},
							unbounded coords=discard,
							every axis plot/.append style={line width=0.9pt},
						]
						\nextgroupplot[title={Throat region}, restrict x to domain=0:10, xmin=0, xmax=10]
						\addplot+[no marks, color1] table[x expr=\thisrow{X}*1000, y=heat_flux_0, col sep=comma] {\momentfilerefined};
						\addplot+[no marks, color1, dashed] table[x expr=\thisrow{X}*1000, y=heat_flux_0, col sep=comma] {\momentfilecoarse};
						\addplot+[no marks, color2] table[x expr=\thisrow{X}*1000, y=heat_flux_translational_0, col sep=comma] {\dsmcfile};
						\nextgroupplot[title={Expansion region}, restrict x to domain=10:113.75, xmin=10, xmax=113.75,
							yticklabel pos=right, ylabel near ticks]
						\addplot+[no marks, color1] table[x expr=\thisrow{X}*1000, y=heat_flux_0, col sep=comma] {\momentfilerefined};
						\addplot+[no marks, color1, dashed] table[x expr=\thisrow{X}*1000, y=heat_flux_0, col sep=comma] {\momentfilecoarse};
						\addplot+[no marks, color2] table[x expr=\thisrow{X}*1000, y=heat_flux_translational_0, col sep=comma] {\dsmcfile};
					\end{groupplot}
				\end{tikzpicture}%
				\hspace*{0.2cm}}
			\caption{Streamwise heat flux $q_x$, split into the near-throat region ($0 \le x \le \SI{10}{\milli\meter}$) and the downstream expansion region ($\SI{10}{\milli\meter} \le x \le \SI{113.75}{\milli\meter}$) to resolve the two disparate magnitude scales.}
			\label{fig:nozzle-argon-centerline-heat-flux}
		\end{subfigure}
		\caption{Fluid dynamic fields along the symmetry line of the nozzle geometry,
			comparing the moment method on a coarse (\num{2608} elements) and a refined
			(\num{10432} elements) triangular mesh against DSMC.
			The strong stochastic fluctuations of the DSMC heat flux in~(c) are characteristic
			of the particle method and stand in contrast to the noise-free moment solution.}
		\label{fig:nozzle-argon-centerline}
	\end{figure}
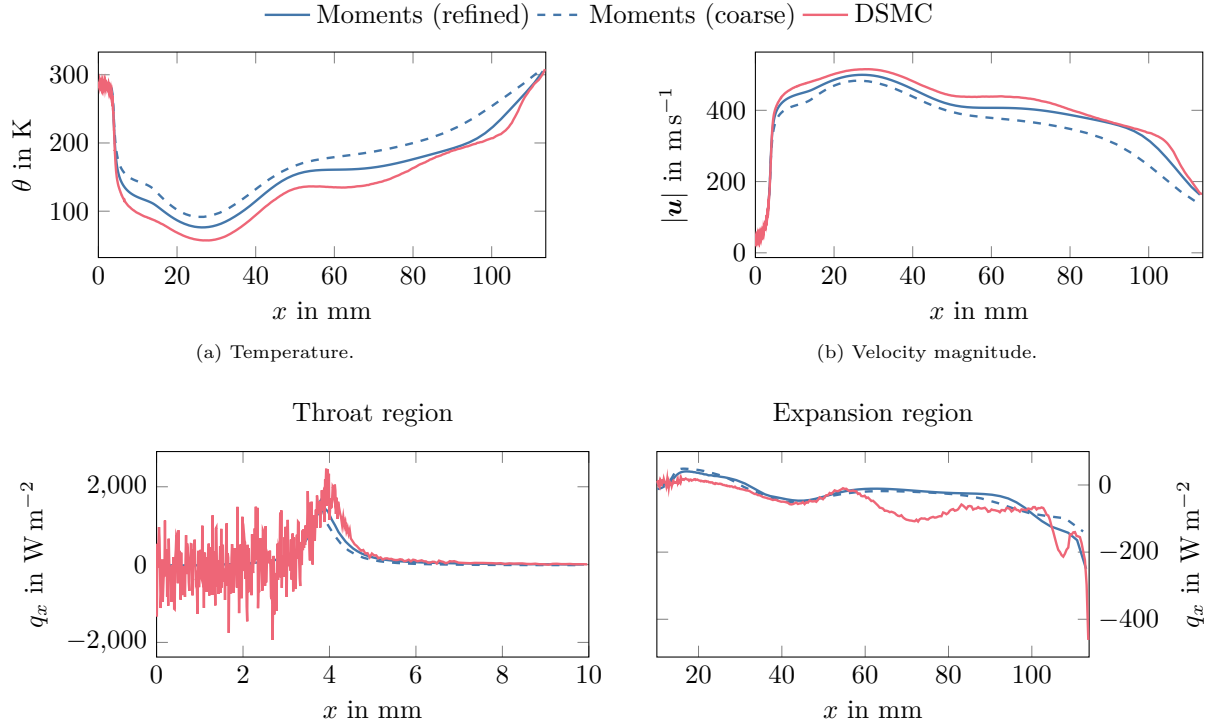

	\subsection{Mass Flow Through a Channel}
	In this example, we analyze the mass flow of a rarefied gas through a long isothermal channel with cross section $\Omega$,
	an open, convex and bounded subset of $\mathbb R^2$, centered around the origin.
	We simulate the gas flow for a wide range of pressures, from the continuum regime to the collisionless limit.
	For the latter, we present analytical results for the distribution function and our quantity of interest,
	the mass flow rate through the channel.
	Finally, we compare the results of our numerical algorithm with experimental measurements for helium in a circular channel.
	This problem has been studied extensively in the literature.
	The experimental and theoretical treatment dates back to Knudsen~\cite{knudsen1909};
	first numerical results for simplified collision models appeared in~\cite{cercignani1966},
	followed by semi-analytical treatments for different cross sections in~\cite{Sharipov1994,Sharipov1999}.
	The work~\cite{Varoutis2008} discusses experimental results for several different cross sections
	and compares them with a simplified computational model.
	Finally, Kunze et al.~\cite{Kunze2022} review results for several channel geometries and gases
	to obtain estimates on molecular diameters.
	To date, however, no results with the full collision operator have been reported.

	First, we assume that the length of the channel $L$ is much larger than the diameter $D$ of the cross section.
	In this case, we can ignore boundary effects at the ends of the channel and prescribe fixed upstream and downstream pressures
$p_1$ and $p_2$, respectively.
	Since the channel has a constant temperature, we assume diffuse reflection boundary conditions of fixed temperature $\theta_0$.
	Instead of simulating the full three-dimensional pipe, we pose a reduced equation
	on the two-dimensional cross section.
	In steady state, our quantity of interest, the mass flow rate, is constant on each cross section.
	Thus, we can capture the correct mass flow rate with a computation on a two-dimensional domain
	if we include the effect of the gradient between the upstream and the downstream pressure.
	To that end, we define the nondimensional pressure gradient $X_p$,
	\[
		X_p = \frac{D}{L} \frac{p_1 - p_2}{p_0},
	\]
	where $p_0$ is the reference pressure used for the nondimensionalization of the Boltzmann equation.
	Following \labelcref{sec:analytical-mass-flow}, we seek the distribution in the form
$f = (p(x_3)/p_0) \, M_{\rho_0, \bm 0, R_s \theta_0} \, g$ with $g \approx 1$,
	so that the perturbation $g$ carries the dependence on the cross-sectional coordinates and the velocity.
	\begin{lemma}
		\label{lem:channel-reduced-be}
		In the limit of small pressure gradients, $X_p \ll 1$, the reduced Boltzmann equation equipped with diffuse reflection boundary conditions
		of unit temperature in nondimensional units reads
		\begin{equation*}
			M_{1, \bm 0, 1} \big(v_1 \partial_{x_1} g
			+ v_2 \partial_{x_2} g \big)
			- X_p M_{1, \bm 0, 1} v_3 = \frac{1}{\kn} {\mathcal C}(M_{1, \bm 0, 1} g),
		\end{equation*}
		with the Knudsen number
		\begin{equation*}
			\kn = \sqrt{\frac{\pi}{2}} \frac{1}{D} \frac{\mu}{p_0} \sqrt{R_s \theta_0}.
		\end{equation*}
		Here, $\mu$ is the dynamic viscosity of the gas at temperature $\theta_0$
		and $R_s$ the specific gas constant.
	\end{lemma}

	For comparison with experimental data, we normalize our result for the mass flow rate.
	The resulting quantity $G$ is called the normalized mass flow rate.

	\begin{definition}
		With the solution $g$ of the reduced Boltzmann equation in \cref{lem:channel-reduced-be},
		the normalized mass flow rate $G$ reads
		\begin{equation*}
			G = \frac{1}{X_p} \frac{1}{|\Omega|} \int_{\Omega} \int_{\mathbb R^3}
			v_3 M_{1, \bm 0, 1}({\bm v}) g(x_1,  x_2, {\bm v})
			\, \dd {\bm v} \, \dd(x_1, x_2).
		\end{equation*}
	\end{definition}

	For the collisionless limit, we obtain an analytical expression for the distribution function.
	\begin{lemma}
		\label{them:pdf-collisionless}
		The solution to the boundary value problem in \cref{lem:channel-reduced-be}
		for a circular cross section of radius $R > 0$ in the collisionless limit
		is given by
		\begin{equation*}
			g(x_1, x_2, \bm v) =
			1 + X_p \frac{v_3}{\sqrt{v_1^2 + v_2^2}} \Bigg(
			\frac{x_1 v_1 + x_2 v_2}{\sqrt{v_1^2 + v_2^2}}
			+ \sqrt{\frac{(x_1 v_1 + x_2 v_2)^2}{v_1^2 + v_2^2} + R^2 - x_1^2 - x_2^2}
			\Bigg)
		\end{equation*}
		for $(x_1, x_2) \in D_R(\bm 0)$ and $\bm v \in \mathbb R^3 \setminus \{ \bm 0 \}$.
	\end{lemma}

	With the analytical solution for the distribution function, we can also compute an analytical value
	for the normalized mass flow rate. This number was already reported by Knudsen~\cite{knudsen1909},
	see \labelcref{sec:analytical-mass-flow} on the relation between the results.

	\begin{lemma}
		\label{thm:mass-flow-rate-collisionless}
		The solution $g$ in \cref{them:pdf-collisionless} generates a mass flux
		with normalized mass flow rate
		\[
			G = \frac{2 \sqrt{2}}{3 \sqrt{\pi}} \approx 0.531923,
		\]
		for the geometry nondimensionalized by the channel diameter, i.e., a cross section of unit diameter, $R = 1/2$.
	\end{lemma}

	For a numerical solution of the reduced model in \cref{lem:channel-reduced-be} we must specify the free
	parameters of the model. These include the length $L$ and diameter $D$ of the channel, the temperature of the walls $\theta_0$
	and the properties of the gas, i.e., its viscosity and molar mass.
	In the following, we use the parameters as reported in \cref{tab:perrier-data}.
	For a sequence of Knudsen numbers, ranging from the continuum to the collisionless regime, we compute the numerical
	solution of the reduced model and the resulting normalized mass flow rate $G$.
	We observe that the mesh size and the degree of the DG basis have minimal effect on $G$.
	However, with increasing Knudsen number, the degree of the moment approximation must be increased.
	This is consistent with our analytical result in \cref{them:pdf-collisionless}
	for the distribution function in the collisionless limit.
	The solution develops a singularity at $\bm v = 0$, so for finite but large Knudsen numbers the distribution
	function exhibits sharp features for small microscopic velocities.
	Hence, resolving these features requires a higher polynomial degree.

	The final result of our computations is shown in \cref{fig:G-flux-experimental}
	with a representative velocity profile at $\kn = 0.44$ displayed in \cref{fig:velocity-profile}.
	The solid line in \cref{fig:G-flux-experimental} representing our results arises from a least squares fit of a cubic polynomial to the
	logarithmically rescaled raw data; see \labelcref{sec:post-processing-mass-flow} for details.
	Overall, our numerical results agree very well with the experimental data.
	For Knudsen numbers below $\kn = 0.2$, our numerical predictions closely follow the experimental data
	up to the (unknown) measurement errors.
	However, beyond the global minimum of the curve, known as the Knudsen minimum, at around $\kn = 1$,
	our values lie systematically above the experimental data.
	Since our numerical result approaches the analytical value for $\kn \to \infty$, we suspect that the experimental
	data has a systematic bias for large Knudsen numbers.
	Instead of increasing after the Knudsen minimum,
	the experimentally measured data decreases for $\kn > 1$, contradicting the predictions of kinetic theory.
	The purpose of \cref{fig:G-flux-experimental} is the comparison against the experimental data
	of~\cite{perrier2011}, which is noisy and carries unknown error bars.
	Readers interested in the raw numerical values may instead consult the mapped data
	in \cref{fig:raw-G-log}.

	\begin{table}
		\caption{Parameters of the flow problem in~\cite{perrier2011}
			for helium in a circular cross section of radius $R$.}
		\label{tab:perrier-data}
		\centering
		\begin{tabular}{@{}cc@{}}
			\toprule
			Parameter  & {Value}                                          \\
			\midrule
			$L$        & \SI{1.82e-2}{\meter}                             \\
			$D$        & \SI{49.6e-6}{\meter}                             \\
			$\theta_0$ & \SI{296.5}{\kelvin}                              \\
			$\mu$      & \SI{1.9732e-05}{\pascal \second}                 \\
			$R_s$      & \SI{2077.26}{\joule \per \kilogram \per \kelvin} \\
			\bottomrule
		\end{tabular}
	\end{table}

	\begin{figure}
		\centering
		\begin{tikzpicture}
			\begin{loglogaxis}
				[
					ymax = 1.0,
					enlargelimits = true,
					scaled ticks = false,
					log ticks with fixed point,
					ticklabel style={
							/pgf/number format/.cd,
							/pgf/number format/fixed
						},
					declare function = {
							c0 = -0.63488;
							c1 = -0.253103;
							c2 = 0.211179;
							c3 = -0.0114429;
							Gtilde(\d) = c0 + c1 * \d + c2 * (\d)^2 + c3 * (\d)^3;
							toDelta(\kn) = sqrt(pi) / 2 / \kn;
							G(\kn) = exp(Gtilde(ln(1 + toDelta(\kn))));
						},
					xlabel = {$\kn$},
					ylabel = {$G$},
					samples = 256,
				]
				\addplot+[only marks] table {data/exp-G.dat};
				\addlegendentry{Experimental Data}
				\addplot+[no markers, thick, domain = 0.01:20] {G(x)};
				\addlegendentry{Numerical Result}
				\addplot+[no markers, thick, dashed, domain = 0.1:20.0] {2 * sqrt(2) / (3 * sqrt(pi))};
				\addlegendentry{Free Stream Limit}
			\end{loglogaxis}
		\end{tikzpicture}
		\caption{Comparison of our numerical result with experimental data.
			Shown is the normalized mass flow rate as a function of the Knudsen number
			for the mass flow problem of helium in a channel with circular cross section.
			The dashed line shows the collisionless limit, $\kn \to \infty$.
			The experimental data is taken from~\cite{perrier2011}.}
		\label{fig:G-flux-experimental}
	\end{figure}
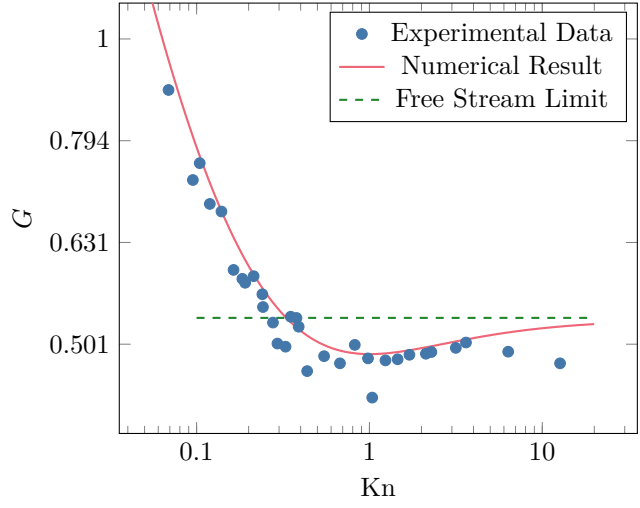

	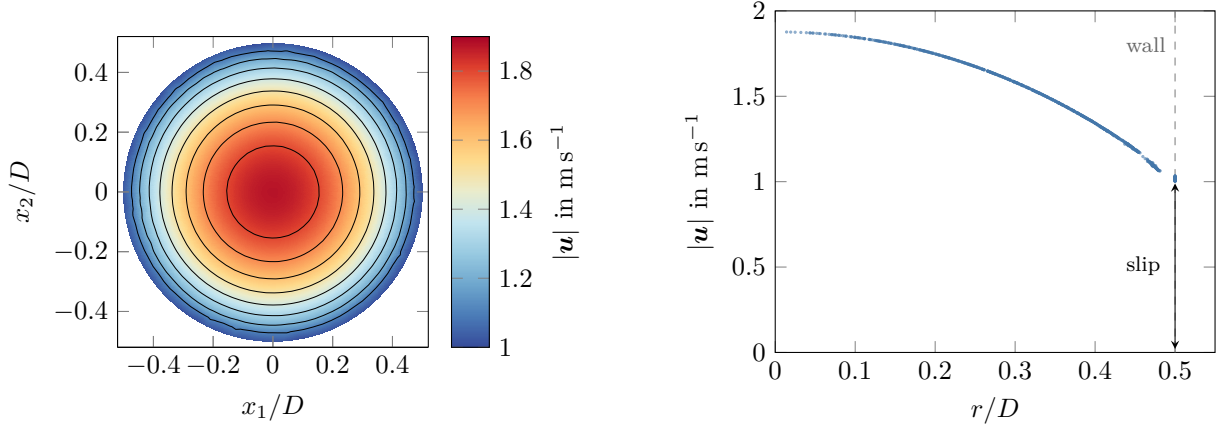
\begin{figure}
		\centering
		\begin{subfigure}[b]{0.48\textwidth}
			\centering
			\begin{tikzpicture}
				\begin{axis}[
						colormap name = mybluered,
						width = 6.6cm,
						view = {0}{90},
						axis equal image,
						xlabel = $x_1 / D$,
						ylabel = $x_2 / D$,
						xmin = -0.52, xmax = 0.52,
						ymin = -0.52, ymax = 0.52,
						colorbar,
						colorbar style = {
								ylabel = {$|\bm u|$ in \unit{\meter\per\second}},
							},
						point meta min = 1.0,
						point meta max = 1.9,
					]
					\addplot3[patch, shader=interp] table {physical_data.dat};
					\addplot3[contour prepared, contour/draw color=black, very thin,
						contour/labels=false] table {data/velocity-contours-kn044.dat};
				\end{axis}
			\end{tikzpicture}
			\caption{Velocity magnitude over the cross-section with iso-velocity
				lines from \SI{1.1}{\meter\per\second} to \SI{1.8}{\meter\per\second}
				in steps of \SI{0.1}{\meter\per\second}.}
			\label{fig:velocity-profile-2d}
		\end{subfigure}
		\hfill
		\begin{subfigure}[b]{0.48\textwidth}
			\centering
			\begin{tikzpicture}
				\begin{axis}[
						width = 7.4cm,
						height = 6.1cm,
						xlabel = {$r / D$},
						ylabel = {$|\bm u|$ in \unit{\meter\per\second}},
						xmin = 0, xmax = 0.55,
						ymin = 0, ymax = 2,
					]
					\addplot[only marks, mark=*, mark size=0.4pt, myblue, opacity=0.6]
					table[x=r, y=u] {data/velocity-radial-kn044.dat};
					\draw[dashed, gray] (axis cs:0.5,0) -- (axis cs:0.5,2)
					node[pos=0.95, anchor=north east, black!60] {\footnotesize wall};
					\draw[<->, >=stealth] (axis cs:0.5,0.02) -- (axis cs:0.5,0.99)
					node[midway, anchor=east, xshift=-2pt] {\footnotesize slip};
				\end{axis}
			\end{tikzpicture}
			\caption{Velocity magnitude at all mesh vertices as a function of the
				distance $r$ from the center of the pipe.}
			\label{fig:velocity-profile-radial}
		\end{subfigure}
		\caption{Velocity profile for $\kn = 0.44$.
			Coordinates are displayed in nondimensional units while the magnitude uses the physical units of the problem.
			The iso-velocity lines in~(a) are concentric circles:
			the profile is rotationally symmetric with the fastest flow along the center line of the pipe.
			Accordingly, the vertex values in~(b) collapse onto a single curve as a function of the radius.
			As expected for a flow with a Knudsen number in the transitional regime,
			there is significant slip at the boundary relative to the maximum speed of the gas:
			the velocity does not drop to zero at the wall.
		}
		\label{fig:velocity-profile}
	\end{figure}

	\subsection{Heat Transfer Between Parallel Walls}

	As a final example, we compute the stationary distribution of the gas between
	two infinite parallel walls with temperatures $\theta_1$ on the left
	and $\theta_2$ on the right.
	In this case, the stationary Boltzmann equation simplifies to
	\begin{equation}\label{eq:parallel-walls}
		v_1 \partial_{x} f(x, \bm v)
		= \frac{1}{\kn} \mathcal C\big(f(x, \cdot)\big)(\bm v),
		\quad (x, \bm v) \in (0, 1) \times \mathbb R^3,
	\end{equation}
	with diffuse reflection boundary conditions of temperatures $\theta_1$ and $\theta_2$
	at $0$ and $1$, respectively.

	Owing to the one-dimensional formulation, important moments such as the bulk velocity $\bm u$
	and the heat flux $\bm q$ are constant.

	As a consequence of the Galilean invariance of the problem statement in
	the $x_2$-$x_3$ plane, the effective dimension of the velocity space reduces from three to two.
	\begin{lemma}\label{lem:invariance-parallel-walls}
		The solution to \eqref{eq:parallel-walls} is of the form
		\begin{equation}
			f(x, \bm v) = g(x, v_1, |\bm v|^2),
		\end{equation}
		that is, it only depends on the first velocity component and the square of its magnitude.
	\end{lemma}

	\begin{lemma}\label{lem:cnst-col-inv}
		For a collision invariant $\psi$ the corresponding flux of the solution $f$
		to~\eqref{eq:parallel-walls},
		\[
			F[\psi](x) = \int_{\mathbb R^3} v_1 \psi(\bm v) f(x, \bm v) \, \dd \bm v,
			\quad x \in (0, 1)
		\]
		is a constant function.
	\end{lemma}

	\begin{proof}
		Multiplying \eqref{eq:parallel-walls} by a collision invariant $\psi$, integrating over $\mathbb R^3$, and using $\int_{\mathbb R^3} \psi(\bm v) \mathcal C\big( f(x, \cdot) \big) \, \dd \bm v = 0$ gives $\frac{\dd}{\dd x} F[\psi](x) = \int_{\mathbb R^3} v_1 \psi(\bm v) \partial_x f(x, \bm v) \, \dd \bm v = 0$, so $F[\psi]$ is constant on $(0, 1)$.
	\end{proof}

	As further consequences we have

	\begin{corollary}\label{cor:zero-velo}
		The bulk velocity $\bm u$ of the stationary solution vanishes.
	\end{corollary}

	\begin{remark}
		The components $u_2 = u_3 = 0$ by \cref{lem:invariance-parallel-walls}, while
		\cref{lem:cnst-col-inv} makes the momentum $\rho u_1$ constant in $x$; since the boundary
		conditions conserve mass, this constant vanishes, and hence $u_1 = 0$ for strictly positive density.
	\end{remark}

	\begin{corollary}\label{cor:zero-heat-flux}
		The heat flux
		\[
			\bm q = \frac{1}{2} \int_{\mathbb R^3} (\bm v - \bm u) |\bm v - \bm u|^2 f(t, \bm x, \bm v)
			\, \dd \bm v
		\]
		of the stationary solution is constant.
		In particular, $q_2 = q_3 = 0$.
	\end{corollary}

	\begin{remark}
		With $\bm u = 0$ from \cref{cor:zero-velo}, the component $q_1$ is the one-dimensional flux of the
		collision invariant $|\bm v|^2$, hence constant by \cref{lem:cnst-col-inv}; the integrands of $q_2, q_3$
		are odd in $(v_2, v_3)$, so $q_2 = q_3 = 0$ by \cref{lem:invariance-parallel-walls}.
	\end{remark}
	First, we analyze the results for the bulk velocity at a fixed Knudsen number of $\kn = 1/10$.
	Results for different Knudsen numbers show the same qualitative behavior.
	By \cref{cor:zero-velo}, the bulk velocity should vanish for the steady-state solution.
	With $p = 0$, the velocity does not vanish in steady state since boundary and interface
	fluxes are imposed weakly.
	As displayed in \cref{fig:convergence-velo-heat-transfer},
	the error decays linearly under mesh refinement, as predicted by general approximation
	results for DG basis functions.
	Furthermore, the error scaling is independent of the velocity discretization, that is, the
	maximal polynomial degree $K$ or the choice of the closure $N$.
	The results improve markedly with $p = 1$.
	In the bulk, the solution converges at optimal second order.
	Near the boundary, however, the solution starts to oscillate,
	resulting in a suboptimal convergence rate under spatial mesh refinement.
	The oscillations do not contradict our stability result in \cref{them:entropy-stability},
	which states stability in an integral sense that does not imply pointwise control over the solution.
	We note, moreover, that \cref{them:entropy-stability} presupposes a single background Maxwellian
	conforming to the imposed boundary data (cf.~\cref{rem:maxwellian-bc}), a hypothesis not strictly
	satisfied by the two-temperature configuration $\theta_1 \neq \theta_2$.
	The integral bound is nonetheless observed to hold:
	the $L_2$-norm of the bulk velocity remains bounded under mesh refinement in space
	and increased polynomial degree in velocity space.

	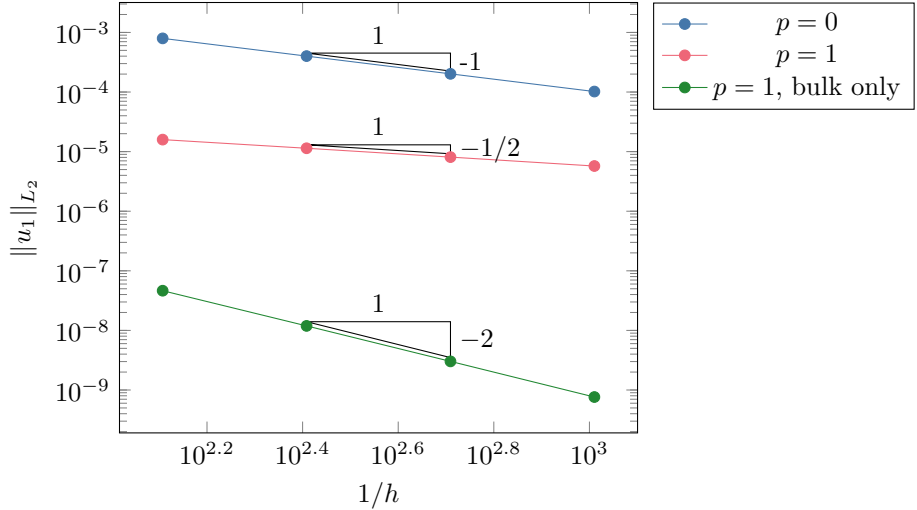
\begin{figure}
		\centering
		\begin{tikzpicture}
			\begin{loglogaxis}[
					xlabel = $1 / h$,
					ylabel = $\|u_1\|_{L_2}$,
					legend style = {legend pos = outer north east},
				]
				\addplot+[mark=*] coordinates {
						(128, 0.000795009)
						(256, 0.000400783)
						(512, 0.000201834)
						(1024, 0.000101886)
					};
				\addlegendentry{$p = 0$}
				\draw (axis cs:256, 4.5e-4) -| (axis cs:512, 2.25e-4) node[above, pos=0.25] {1} node[right, pos=0.75] {-1}  -- cycle;

				\addplot+[mark=*] coordinates {
						(128, 1.59035e-05)
						(256, 1.13858e-05)
						(512, 8.09144e-06)
						(1024, 5.73328e-06)
					};
				\addlegendentry{$p = 1$}
				\draw (axis cs:256, 1.3e-5) -| (axis cs:512, 0.92e-5) node[above, pos=0.25] {1} node[right, pos=0.75] {$-1/2$}  -- cycle;

				\addplot+[mark=*] coordinates {
						(128, 4.63625e-08)
						(256, 1.19023e-08)
						(512, 3.02159e-09)
						(1024, 7.62284e-10)
					};
				\addlegendentry{$p = 1$, bulk only}
				\draw (axis cs:256, 1.4e-8) -| (axis cs:512, 0.35e-8) node[above, pos=0.25] {1} node[right, pos=0.75] {$-2$}  -- cycle;
			\end{loglogaxis}
		\end{tikzpicture}
		\caption{Convergence of the bulk velocity $u_1$ for the DGFE moment method with $\kn=1/10$ under uniform mesh refinement in the spatial domain.
			The velocity discretization is fixed with $N = 1$ and $K = 9$.
			The norm in the bulk is computed on the interval $(1/8, 7/8)$, i.e., it excludes the boundary layer.}
		\label{fig:convergence-velo-heat-transfer}
	\end{figure}
	To investigate the heat flux as a function of the Knudsen number,
	we numerically validate the linear theory in~\cite{sone2007} for small temperature gradients.
	Its prediction, also known as the Sherman--Lees formula, gives the constant heat flux in steady state as
	\begin{equation}\label{eq:sherman-lees}
		q_1 = \frac{1}{c_0 + c_1 / \kn},
	\end{equation}
	for two constants $c_0$ and $c_1$, cf.~\cite[Equation~(4.18c)]{sone2007}.
	This relationship has been well tested in experimental studies; see, for instance,~\cite{trott2011experimental}
	and the references cited therein.
	The Sherman--Lees formula~\eqref{eq:sherman-lees} can be recast as
	\begin{equation}
		\frac{1}{q_1} = c_0 + c_1 \frac{1}{\kn},
	\end{equation}
	which reveals that $1 / q_1$ depends linearly on $1 / \kn$.
	We use this relation to infer the values of $c_0$ and $c_1$
	by performing a linear least squares fit to our numerical data.
	For $\theta_1 = 1.25$ and $\theta_2 = 1$,
	the least squares fit over a range of Knudsen numbers from $\num{e-3}$ to $\num{e3}$
	for a hard sphere gas yields
	\begin{equation}\label{eq:coeff-sherman-lees}
		c_0 = \num{5.36159}, \quad c_1 = \num{1.26286},
	\end{equation}
	with coefficient of alienation $1 - R^2 = \num{3.82e-8}$.
	Our numerical data together with the Sherman--Lees prediction is displayed in \cref{fig:heat-flux}.
	Both agree very well over the full range of Knudsen numbers,
	with a slight deviation for large Knudsen numbers.
	This deviation is explained by the nonlinear transformation of the numerical data:
	large Knudsen numbers accumulate at $0$, increasing the sensitivity
	of the linear least squares fit.

	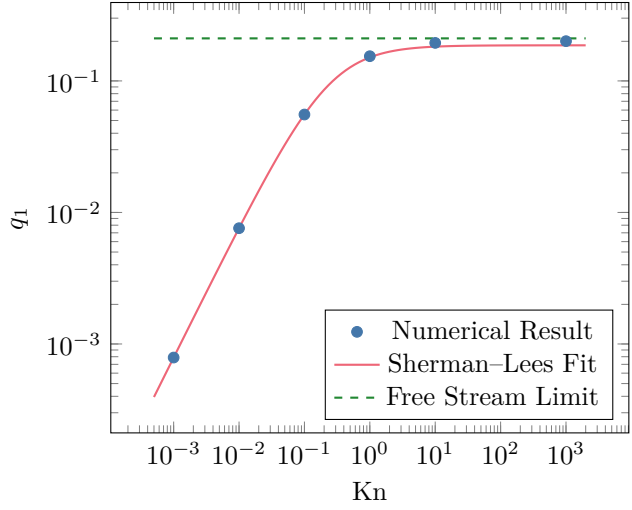
\begin{figure}
		\centering
		\begin{tikzpicture}
			\begin{loglogaxis}
				[
					xlabel = {$\kn$},
					ylabel = $q_1$,
					legend pos = south east,
					declare function = {
							c0 = 5.36159;
							c1 = 1.26286;
							Qtilde(\x) = c1 * \x + c0;
							Q(\kn) = 1.0 / Qtilde(1.0 / \kn);
							FreeStream = 0.21058728585439238;
						},
				]
				\addplot+[mark=*, only marks] coordinates {
						(0.001,0.0007885111844539124)
						(0.01,0.007591175010632328)
						(0.1,0.05552309459479234)
						(1,0.15412853423268813)
						(10,0.19453747977625657)
						(1000,0.2005535316095865)
					};
				\addlegendentry{Numerical Result}
				\addplot+[domain = 5e-4:2e3, thick] { Q(x) };
				\addlegendentry{Sherman--Lees Fit}
				\addplot+[domain = 5e-4:2e3, dashed, thick] { FreeStream };
				\addlegendentry{Free Stream Limit}

			\end{loglogaxis}
		\end{tikzpicture}
		\caption{Numerical data on the heat flux between parallel walls of normalized temperatures $\theta_1 = 1.25$ and $\theta_2 = 1$.
			The parameters for the Sherman--Lees formula~\eqref{eq:sherman-lees} are given in~\eqref{eq:coeff-sherman-lees}.}
		\label{fig:heat-flux}
	\end{figure}

	\section{Conclusions and Outlook}\label{sec:outlook}
	In this work, we have developed and analyzed a robust Galerkin methodology for the Boltzmann equation equipped with bilinear collision operators. By utilizing $\varphi$-divergences for moment closures, we formulated a sequence of families of hierarchies of fluid dynamic equations that rigorously bridge the kinetic and macroscopic fluid dynamic regimes. While such closures intrinsically inherit fundamental properties, such as Galilean symmetries and conservation laws, from the continuous Boltzmann equation, the exact entropy production inequality may not be satisfied. To explicitly quantify and address this thermodynamic discrepancy, we introduced a convergent sequence of approximate collision operators that satisfy an approximate entropy production inequality. Moreover, we proved that these constructed operators strictly preserve Galilean symmetries, exact conservation laws, and the correct physical relaxation towards local thermodynamic equilibrium. Our results mathematically establish that $\varphi$-divergences serve as valid approximate entropies for solutions of \eqref{eq:ibvBE}.

	Furthermore, we demonstrated that, conditioned on satisfying an entropy production inequality, the $\varphi$-divergence moment equations form symmetric-dissipative hyperbolic systems, ensuring local-in-time well-posedness. Applying a fully implicit space-time discontinuous Galerkin (DG) finite element discretization yields a numerical scheme that is rigorously entropy-stable with respect to the associated approximate operator $\mathcal C_N$ (the reported computations employ the full operator $\mathcal C$, cf.~\cref{sec:approx_coll}). We validated our deterministic solver against established analytical benchmarks and experimental data, focusing on low-speed, rarefied flow regimes such as isothermal channel flows and heat transfer between parallel plates. Beyond these low-speed settings, we computed the supersonic nozzle flow of argon and validated the moment approximation against a DSMC reference with identical collision-model parameters. For this strongly nonequilibrium expansion, the nonlinear renormalization proves essential: the linear closure fails to converge to a steady state, whereas the nonlinear closure converges robustly.
	Looking forward, we aim to extend this framework to further high Mach number test cases, building on the supersonic nozzle flow of argon presented in this work. In such highly non-equilibrium regimes, deploying nonlinear approximations of the exponential renormalization will be essential for maintaining the stability and accuracy of the numerical approximation. Additionally, future research will integrate residual-based stabilization techniques to mitigate localized oscillations near flow discontinuities and to improve pointwise control over higher-order discrete polynomial solutions.

	\section*{Acknowledgements}
	TK acknowledges funding received from the European Union's Horizon 2020 research and innovation programme under the Marie Sk\l{}odowska-Curie grant agreement No 899987. This work is also supported by the 14AMI project of the Chips Joint Undertaking and its members, including the top-up funding by RVO (The Netherlands Enterprise Agency).

	\appendix
	\section{Details on Approximate Collision Operators}\label{sec:details-approximate-collision}
	The existence and uniqueness of the Lagrange parameters $\bm \lambda$ in the reconstruction problem
	follow from the finite-dimensional Minty--Browder theorem~\cite[Theorem 9.14-1]{ciarlet2013},
	whereby a continuous, coercive and strictly monotone map $F : \mathbb R^m \to \mathbb R^m$ is bijective.

	\begin{proof}[Proof of \cref{thm:reconstruction-problem}]
		We prove that the function
		\begin{equation}
			F: \mathbb R^m \to \mathbb R^m,~\bm \mu \mapsto \int_{\mathbb R^3} \bm m \beta_N(\bm \mu \cdot \bm m) \, \dd \bm v
		\end{equation}
		satisfies the assumptions of the Minty--Browder theorem.
		First, we note that $F$ is continuous and strictly monotone,
		\begin{equation}
			(\bm \mu - \bm \lambda) \cdot \big( F(\bm \mu) - F(\bm \lambda) \big)
			=
			\int_{\mathbb R^3} M \big( \bm \mu \cdot \bm m - \bm \lambda \cdot \bm m \big) \left\{
			\left( 1 + \frac{\bm \mu \cdot \bm m}{N} \right)^N - \left( 1 + \frac{\bm \lambda \cdot \bm m}{N} \right)^N
			\right\} \,
			\dd \bm v \geq 0,
		\end{equation}
		since $N$ is odd and thus $z \mapsto (1 + z/N)^N$ is strictly increasing, so equality holds if and only if $\bm \mu = \bm \lambda$.
		To show that $F$ is coercive, we expand $(1 + \bm \mu \cdot \bm m / N)^N$ binomially. With $\bm e_{\bm \mu} = \bm \mu / |\bm \mu|$, the leading term in $|\bm \mu|$ is
		\begin{equation}
			\frac{|\bm \mu|^{N}}{N^N} \int_{\mathbb R^3} (\bm e_{\bm \mu} \cdot \bm m )^{N + 1} M \, \dd \bm v,
		\end{equation}
		whose coefficient $\int_{\mathbb R^3} (\bm e \cdot \bm m )^{N + 1} M \, \dd \bm v$ is strictly positive for every $\bm e \in \mathbb S^{m - 1}$ since $N$ is odd, with the infimum attained on the compact sphere $\mathbb S^{m - 1}$.
		As the remaining terms are $O(|\bm \mu|^{N-1})$, it follows that $\bm \mu \cdot F(\bm \mu) / |\bm \mu| \to \infty$ as $|\bm \mu|$ diverges, implying that $F$ is coercive.
	\end{proof}

	To relate $P_N$ to the usual product of two real numbers, we use the fundamental lemma for the
	exponential function~\citep[p.\,104]{koenigsberger2004}: for $z \in \mathbb R$,
	\begin{equation}\label{eq:exponential-limit-zero}
		\lim_{N \to \infty} \left( \frac{z}{N} + o(1/N) \right)^N = 0,
		\qquad
		\lim_{N \to \infty} \left( 1 + \frac{z}{N} + o(1/N) \right)^N = e^z,
	\end{equation}
	both uniformly on compact sets.

	\begin{lemma}\label{lem:ConvergenceP_N}
		For $z_1, z_2 \geq 0$ not both equal to zero,
		\[
			\lim\limits_{N \to \infty} P_N(z_1, z_2) = z_1 z_2.
		\]
		The limit is uniform on compact subsets of $(0, \infty) \times (0, \infty)$.
	\end{lemma}
	\begin{proof}
		The linearization $z^{\frac{1}{N}} = 1 + (\log z)/N + o(1/N)$ as $N \to \infty$ (valid for $z > 0$) gives
		\begin{equation}
			P_N(z_1, z_2) = \big(z_1^{\frac{1}{N}} + z_2^{\frac{1}{N}} - 1\big)^N
			= \left( 1 + \frac{\log z_1 + \log z_2}{N} + o(1 / N) \right)^N.
		\end{equation}
		If $z_1, z_2 > 0$, the second limit in~\eqref{eq:exponential-limit-zero} yields $P_N(z_1, z_2) \to e^{\log z_1 + \log z_2} = z_1 z_2$; if instead $z_2 = 0$ and $z_1 > 0$, then $z_2^{\frac{1}{N}} = 0$ cancels the $-1$, leaving $P_N(z_1, z_2) = (\log z_1 / N + o(1/N))^N \to 0 = z_1 z_2$ by the first limit in~\eqref{eq:exponential-limit-zero}.
		Since the second limit in~\eqref{eq:exponential-limit-zero} is uniform on compact sets,
		the convergence is uniform on compact subsets of $(0, \infty) \times (0, \infty)$.
	\end{proof}
	Weak forms of approximate collision operators follow those of the bilinear collision operator:
	\begin{equation}\label{def:IncompleteWeak}
		\begin{aligned}
			\int_{\mathbb R^3} \psi \, \mathcal{C}_N(f) \, \dd \bm v & =
			\begin{aligned}[t]
				\phantom{-}\frac{1}{2} \int_{\mathbb R^3} \int_{\mathbb R^3} \int_{\mathbb S^2}
				B(\bm v - \bm v_*, \bm \sigma) \background \, \background_* \,
				P_N\left( \frac{f}{\background}, \frac{f_*}{\background_*} \right)
				\Big(
				\acute \psi + \acute \psi_* - \psi - \psi_*
				\Big)
				\, \dd \bm \sigma \, \dd \bm v_* \, \dd \bm v
			\end{aligned} \\
			                                                         & =
			\begin{aligned}[t]
				-\frac{1}{4} \int_{\mathbb R^3} \int_{\mathbb R^3} \int_{\mathbb S^2}
				 & B(\bm v - \bm v_*, \bm \sigma) \background \, \background_*
				\left( P_N\left( \frac{\acute f}{\acute{\background}}, \frac{\acute f_*}{\acute{\background}_*} \right) - P_N\left( \frac{f}{\background}, \frac{f_*}{\background_*} \right) \right) \\
				 & \times \Big(
				\acute \psi + \acute \psi_* - \psi - \psi_*
				\Big)
				\, \dd \bm \sigma \, \dd \bm v_* \, \dd \bm v.
			\end{aligned}
		\end{aligned}
	\end{equation}

	For $f \in \mathscr D(\mathcal C)$, $\mathcal C_N(f)$ possesses moments of arbitrary order.
	This is a consequence of the following lemma.

	\begin{lemma}\label{lem:BoundP_N}
		For $z_1, z_2 \geq 0$ and $N \in \mathbb N$ odd,
		\begin{equation}
			|P_N(z_1, z_2)| \leq \max\{z_1, 1 \} \max\{z_2, 1\}.
		\end{equation}

		\begin{proof}
			Since the left- and the right-hand side of the inequality are symmetric in their arguments,
			it suffices to prove the inequality for the three cases, $z_1, z_2 \geq 1$,  $z_1, z_2 \in [0, 1]$ and $z_1 \in [0,1], z_2 \geq 1$.
			We begin with the last case.
			Then,
			\begin{equation}
				0 \leq P_N(z_1, z_2) = \big(z_1^{\frac{1}{N}} + z_2^{\frac{1}{N}} - 1\big)^N \leq \big(1 + z_2^{\frac{1}{N}} - 1\big)^N = z_2.
			\end{equation}
			For $z_1, z_2 \in [0,1]$,
			\begin{equation}
				P_N(z_1, z_2) \leq (1 + 1 - 1)^N = 1,
			\end{equation}
			and, as $N$ is odd,
			\begin{equation}
				P_N(z_1, z_2) \geq (0 + 0 - 1)^N = -1.
			\end{equation}
			For $z_1, z_2 \geq 1$, factoring out $z_1 z_2$ gives $P_N(z_1, z_2) = z_1 z_2 \big( (1/z_1)^{1/N} + (1/z_2)^{1/N} - (1/(z_1 z_2))^{1/N} \big)^N$, so with $a = z_1$, $b = z_2 \geq 1$ the bound $P_N(z_1, z_2) \leq z_1 z_2$ follows from
			\[
				\frac{1}{a} + \frac{1}{b} - \frac{1}{a b}
				= 1 - \Big( 1 - \frac{1}{a} \Big)\Big( 1 - \frac{1}{b} \Big) \leq 1.
			\]
		\end{proof}
	\end{lemma}
	Therefore, the domain and codomain of $\mathcal C_N$ are well-defined, since by \cref{lem:BoundP_N},
	\begin{equation}
		\left| \background \background_*
		P_N\left( \frac{f}{\background}, \frac{f_*}{\background_*} \right) \right| \leq \max\{\background, f \} \max\{\background_*, f_* \}
	\end{equation}
	and both $f$ and $\background$ lie in $\mathscr D(\mathcal C)$.

	We now turn to the proofs of the main statements in \cref{sec:approx_coll}.
	\begin{proof}[Proof of \cref{thm:galilean-approximate-collision}]
		Applying $\mathsf T$ to $\mathcal C_N(f)$, substituting $\bm v_* \mapsto \bm v_* - \bm u$, and using $\mathsf T B = B$ for the collision weight and $\mathsf T M = M$ for the background gives
		\begin{align*}
			\mathsf T\mathcal C_N(f) = \int\limits_{\mathbb R^3}\int\limits_{\mathbb S^2} B\, \background \, \background_*
			\left( P_N\left( \frac{\mathsf T\acute f}{\acute{\background}}, \frac{\mathsf T \acute f_*}{\acute{\background}_*} \right) - P_N\left( \frac{\mathsf Tf}{\background}, \frac{\mathsf Tf_*}{\background_*} \right) \right)
			\, \dd \bm \sigma \, \dd \bm v_* = \mathcal C_N(\mathsf Tf),
		\end{align*}
		thus proving the desired equality.
	\end{proof}

	To characterize the collision invariants of $\mathcal C_N$, we follow the strategy used in \citep{cercignani1990} to determine the collision invariants of $\mathcal C$, which requires computing the linearized collision operator. To linearize $\mathcal C_N$ about $M$, we need the following lemma concerning the linearization of $P_N$ at $(1,1)$.
	\begin{lemma}\label{lem:LinearizationP_N}
		For small $h_1, h_2 \in \mathbb R$,
		\[
			P_N(1 + h_1, 1 + h_2) = 1 + h_1 + h_2 + o\big(h_1, h_2\big).
		\]
	\end{lemma}

	\begin{proof}
		One readily checks that $P_N(1,1) = 1$ and $\nabla P_N(1,1) = (1,1)^{\top}$ from which the assertion follows.
	\end{proof}

	\begin{proof}[Proof of \cref{thm:linearization-approximate-collision}]
		We write $f = \background + s g$ with $s \in \mathbb R$ and $g \in \mathscr D(\mathcal C)$.
		By \cref{lem:LinearizationP_N},
		\[
			\background \background_* P_N\left(\frac{f}{\background}, \frac{f_*}{\background_*} \right) = \background \background_* + s \background g_* + s \background_* g + o(s),
		\]
		so that for a test function $\psi : \mathbb R^3 \to \mathbb R$, using the weak form of the Boltzmann collision operator, the definition of $\mathcal L$, and $\int_{\mathbb R^3} \psi \, \mathcal C(\background) \, \dd \bm v = 0$ since Maxwellians reside in the kernel of $\mathcal C$,
		\begin{equation}
			\int_{\mathbb R^3} \psi \, \mathcal C_N(f) \, \dd \bm v
			=
			\int_{\mathbb R^3} \psi \, \mathcal C_N(\background + s g) \, \dd \bm v
			=
			s \int_{\mathbb R^3} \psi \, \mathcal L g \, \dd \bm v
			+
			o(s),
		\end{equation}
		from which we directly identify $\mathcal L$ as the linearization of $\mathcal C_N$ around $\background$.
	\end{proof}

	With the linearization of $\mathcal C_N$ at hand,
	we turn to the characterization of its invariants.

	\begin{proof}[Proof of \cref{thm:conservation-approximate-collision}]
		By the first form of \cref{def:IncompleteWeak}, $\int_{\mathbb R^3} \psi \, \mathcal{C}_N(f) \, \dd \bm v$ is a symmetrized integral of $(\acute \psi + \acute \psi_* - \psi - \psi_*)$, which vanishes for all $f \in \mathscr{D}(\mathcal C)$ whenever $\psi$ is a collision invariant.
		To prove that $\mathscr{I}$ includes all collision invariants of $\mathcal C_N$, let $\psi: \mathbb R^3 \to \mathbb R$
		be a test function with
		\begin{equation}
			\int_{\mathbb R^3} \psi \, \mathcal C_N(f) \, \dd \bm v = 0 \quad \forall f \in \mathscr{D}(\mathcal C).
		\end{equation}
		Choosing $f = \background + s g$ with a small parameter $s \in \mathbb R$ and $g \in \mathscr D(\mathcal C)$, we obtain with \cref{thm:linearization-approximate-collision},
		\[
			0 = \int_{\mathbb R^3} \psi \, \mathcal C_N(f) \, \dd \bm v = s \int_{\mathbb R^3} \psi \, \mathcal L g \, \dd \bm v + o(s).
		\]
		Dividing by $s$ and taking the limit $s \to 0$ while noting that $o(s) / s \to 0$ yields
		\[
			\int_{\mathbb R^3} \psi \, \mathcal L g \, \dd \bm v = 0
		\]
		for all $g \in \mathscr D(\mathcal C)$.
		This means that $\psi$ is a collision invariant of the linearized collision operator.
		It is well known~\citep{cercignani1990} that this implies $\psi \in \mathscr I$
		and hence shows the asserted equivalence.
	\end{proof}

	\begin{proof}[Proof of \cref{thm:dissipation-approximate-collision}]
		With the second weak form of $\mathcal C_N$ in \eqref{def:IncompleteWeak} we obtain
		\begin{align}
			\int_{\mathbb R^3} \varphi_N'\left( \frac{f}{\background} \right) \, \mathcal C_N(f) \, \dd \bm v & =
			\begin{aligned}[t]
				-\frac{1}{4} \int_{\mathbb R^3} \int_{\mathbb R^3} \int_{\mathbb S^2}
				 & B(\bm v - \bm v_*, \bm \sigma) \background \background_*                                                                                                                                    \\
				 & \times \left[
					\varphi_N'\left(\frac{\acute f}{\acute{\background}}\right) + \varphi_N'\left(\frac{\acute f_*}{\acute{\background}_*}\right) - \varphi_N'\left(\frac{f}{\background}\right) - \varphi_N'\left(\frac{f_*}{\background_*}\right)
				\right]                                                                                                                                                                                        \\
				 & \times \left[ P_N\left( \frac{\acute f}{\acute{\background}}, \frac{\acute f_*}{\acute{\background}_*} \right) - P_N\left( \frac{f}{\background}, \frac{f_*}{\background_*} \right) \right]
				\, \dd \bm \sigma \, \dd \bm v_* \, \dd \bm v
			\end{aligned} \\
			                                                                                                  & =
			\begin{aligned}[t]
				-\frac{1}{4} \int_{\mathbb R^3} \int_{\mathbb R^3} \int_{\mathbb S^2}
				 & B(\bm v - \bm v_*, \bm \sigma) \background(\bm v) \background(\bm v_*)                                                                                                                      \\
				 & \times \left[
					\varphi_N'\left(P_N\left(\frac{\acute f}{\acute {\background}}, \frac{\acute f_*}{\acute{\background}_*}\right)\right) - \varphi_N'\left(P_N\left(\frac{f}{\background}, \frac{f_*}{\background_*}\right)\right)
				\right]                                                                                                                                                                                        \\
				 & \times \left[ P_N\left( \frac{\acute f}{\acute{\background}}, \frac{\acute f_*}{\acute{\background}_*} \right) - P_N\left( \frac{f}{\background}, \frac{f_*}{\background_*} \right) \right]
				\, \dd \bm \sigma \, \dd \bm v_* \, \dd \bm v,
			\end{aligned}\label{eq:weak-dissipation}
		\end{align}
		where we have used the functional equation~\eqref{eq:functionalEquationPhi_N} in the second step.
		Since $\varphi_N$ is convex, its first derivative is monotone,
		\begin{equation}
			\big( \varphi_N'(z_2) - \varphi_N'(z_1) \big)(z_2 - z_1) \geq 0, \quad z_1, z_2 \in \mathbb R,
		\end{equation}
		so the integrand in~\eqref{eq:weak-dissipation} is nonnegative.
		Therefore,
		\[
			\int_{\mathbb R^3} \varphi_N'\left( \frac{f}{\background} \right) \, \mathcal C_N(f) \, \dd \bm v \leq 0.
		\]
		The integral vanishes if and only if the integrand in~\eqref{eq:weak-dissipation} is zero.
		Owing to $B > 0$ and $M > 0$ and the injectivity of $\varphi_N'$, this is equivalent to
		\begin{equation}
			\varphi_N'\left( \frac{\acute f}{\acute{\background}} \right)
			+ \varphi_N'\left( \frac{\acute f_*}{\acute{\background}_*} \right)
			=
			\varphi_N'\left( \frac{f}{\background} \right)
			+ \varphi_N'\left( \frac{f_*}{\background_*} \right).
		\end{equation}
		Hence, $\varphi_N'(f / M)$ must be a collision invariant.
	\end{proof}

	Finally, we prove the weak convergence of the approximate collision operators
	to Boltzmann's collision operator.

	\begin{proof}[Proof of \cref{thm:convergence-approximate-collision}]
		By \cref{lem:BoundP_N}, $\left| \background \background_* P_N( f/\background, f_*/\background_* ) \right| \leq \max\{\background, f \} \max\{\background_*, f_* \}$, so the integrand in $\int_{\mathbb R^3} \psi \, \mathcal C_N(f) \, \dd \bm v$ is bounded independently of $N$.
		Since
		\begin{equation}
			\lim_{N \to \infty} P_N(z_1, z_2) = z_1 z_2
		\end{equation}
		for all $z_1, z_2 > 0$ by \cref{lem:ConvergenceP_N},
		\begin{equation}
			\background \background_* P_N\left( \frac{f}{\background}, \frac{f_*}{\background_*} \right) \to f f_*, \quad N \to \infty,
		\end{equation}
		almost everywhere.
		The claim thus follows from the dominated convergence theorem.
	\end{proof}

	\section{Details on the Discontinuous Galerkin Finite Element Approximation}
	\label{sec:details-dgfem}
	If the background Maxwellian $M$ fulfills the imposed boundary condition,
	we can use the Darrozès--Guiraud inequality from \Cref{lem:Darrozes_Guiraud}
	to bound the entropy production at the boundary.

	\begin{proof}[Proof of \cref{them:entropy-stability}]
		Our strategy for the proof is as follows.
		First, we set $\bm w = \bm\lambda^{h,p}$ in the variational formulation~\eqref{eq:DGform}, define $g:=\bm\lambda^{h,p}\cdot\bm m$ and note that
		\begin{equation}\label{eq:Galduality}
			g= \varphi_N'\left( \frac{\beta_N(g)}{\background} \right).
		\end{equation}
		\Cref{thm:dissipation-approximate-collision} implies that
		\begin{equation}\label{eq:sest}
			s_N(\bm\lambda^{h,p}; \bm\lambda^{h,p}) \leq 0.
		\end{equation}
		Therefore, if we can prove that $a_x^{\Omega}(\bm\lambda^{h,p}; \bm\lambda^{h,p})+ a_x^{\partial\Omega}(\bm\lambda^{h,p}; \bm\lambda^{h,p})\geq 0$, then $a_t(g; g) \leq 0$, from which the assertion follows.

		To estimate $a^{\Omega}_x(\bm\lambda^{h,p},\bm\lambda^{h,p})$, we note that the positivity assumption on the dissipation flux in \eqref{eq:D>0} implies that
		\begin{equation}\label{eq:DGDineq}
			a^{\Omega}_x(\bm\lambda^{h,p},\bm\lambda^{h,p}) \geq
			\sum_n
			\sum_{e\in\mathcal{I}^h}
			\int\limits_{\mathrlap{I^n \times e \times \mathbb R^3 \times [0,1]}}
			\bm v \cdot \jump{g } \,
			\beta_N(\hat{g}(\zeta))
			\, \dd\zeta \, \dd\bm v \, \dd\bm s \, \dd t
			-
			\sum_n
			\sum_{\kappa\in\mathcal{T}^h}
			\int\limits_{\mathrlap{ I^n\times \kappa\times \mathbb R^3}}
			\beta_N(g)
			\,
			\bm v \cdot\partial_{\bm x}g
			\, \dd\bm v \, \dd\bm x \, \dd t.
		\end{equation}
		Using the chain rule in reverse, we rewrite the integrand of the volumetric integral in \eqref{eq:DGDineq}:
		\begin{equation}\label{eq:revchainr}
			\beta_N(g)
			\,
			\bm v \cdot\partial_{\bm x}g
			=
			\bm v \cdot\partial_{\bm x}M\varphi^*_N(g),
		\end{equation}
		where $M\varphi_N^*(\cdot)$ denotes the antiderivative of $\beta_N(\cdot)$ and corresponds to the convex conjugate of $\varphi_N(\cdot)$:
		\begin{equation}\label{eq:convxconj}
			\varphi^*_N(s) = \sup_z \{sz - \varphi_N(z)\}.
		\end{equation}
		Substituting \eqref{eq:revchainr} into \eqref{eq:DGDineq} and using the divergence theorem yields
		\begin{multline}\label{eq:DGDdiv}
			a^{\Omega}_x(\bm\lambda^{h,p},\bm\lambda^{h,p}) \geq
			\sum_n
			\sum_{e\in\mathcal{I}^h}
			\int\limits_{\mathrlap{I^n \times e \times \mathbb R^3 \times [0,1]}}
			\bm v \cdot \jump{g} \,
			\beta_N(\hat{g}(\zeta))
			\, \dd\zeta \, \dd\bm v \, \dd\bm s \, \dd t
			-
			\sum_n
			\sum_{e\in\mathcal{S}^h}
			\int\limits_{\mathrlap{I^n \times e \times \mathbb R^3}}
			\bm v \cdot \jump{\varphi^*_N(\bm \lambda^{h,p} \cdot \bm m)}
			\, \dd\bm v \, \dd\bm s \, \dd t
			\\
			=
			-
			\sum_n
			\sum_{e\in\mathcal{G}^h}
			\int\limits_{\mathrlap{I^n \times e \times \mathbb R^3}}
			(\bm v \cdot \bm n) \, M\varphi^*_N(\bm \lambda^{h,p} \cdot \bm m)
			\, \dd\bm v \, \dd\bm s \, \dd t,
		\end{multline}
		where the sum over $e\in\mathcal I$ vanishes upon using the fundamental theorem of calculus to express $\jump{\varphi^*_N(\bm\lambda^{h,p}\cdot\bm m)}$ as a line integral and noting from \eqref{eq:convxconj} that $(\varphi^*_N)'(\cdot)=\beta_N(\cdot)$.

		To show that $a_x^{\Omega}(\bm\lambda^{h,p},\bm\lambda^{h,p}) + a_x^{\partial\Omega}(\bm\lambda^{h,p},\bm\lambda^{h,p})\geq0$, we use the estimate in \eqref{eq:DGDdiv} to note that
		\begin{multline}\label{eq:axineq}
			a_x^{\Omega}(\bm\lambda^{h,p},\bm\lambda^{h,p}) + a_x^{\partial\Omega}(\bm\lambda^{h,p},\bm\lambda^{h,p})\geq
			\sum_n
			\sum_{e\in\mathcal{G}^h}
			\int\limits_{\mathrlap{I^n \times e \times \mathbb R^3_{\text{out}}}}
			\bm v \cdot \bm n\,g\beta_N(g)
			\, \dd\bm v \, \dd\bm s \, \dd t
			\\
			+
			\sum_n
			\sum_{e\in\mathcal{G}^h}
			\int\limits_{\mathrlap{I^n \times e \times \mathbb R^3_{\text{in}}}}
			(\bm v \cdot \bm n)\, g\beta_N(g_w)
			\, \dd\bm v \, \dd\bm s \, \dd t
			-
			\sum_n
			\sum_{e\in\mathcal{G}^h}
			\int\limits_{\mathrlap{I^n \times e \times \mathbb R^3}}
			\bm v \cdot \bm n \,M \varphi^*_N(g)
			\, \dd\bm v \, \dd\bm s \, \dd t.
		\end{multline}
		Expressing $\varphi_N^*(\cdot)$ in terms of its convex conjugate $\varphi_N(\cdot)$ using \eqref{eq:convxconj}, and decomposing the velocity integral over $\mathbb R^3$ into $\mathbb R^3_{\text{in}}$ and $\mathbb R^3_{\text{out}}$, we have
		\begin{equation}\label{eq:bndent}
			\int\limits_{\mathbb R^3}
			(\bm v \cdot \bm n) \, M \varphi^*_N(g)
			\, \dd\bm v
			=
			\left(
			\int_{\mathbb R^3_{\text{out}}}
			+
			\int_{\mathbb R^3_{\text{in}}}
			\right)
			\bm v \cdot \bm n\, \left(g\beta_N(g)
			- M \varphi_N\left( \frac{\beta_N(g)}{M} \right) \right)
			\, \dd\bm v.
		\end{equation}
		Substituting \eqref{eq:Galduality} and \eqref{eq:bndent} into \eqref{eq:axineq} yields
		\begin{multline}\label{eq:axconvxconj}
			a_x^{\Omega}(\bm\lambda^{h,p},\bm\lambda^{h,p}) + a_x^{\partial\Omega}(\bm\lambda^{h,p},\bm\lambda^{h,p})
			\geq
			\\
			\sum_n
			\sum_{e\in\mathcal{G}^h}
			\int\limits_{\mathrlap{I^n \times e \times \mathbb R^3_{\text{in}}}}
			\bm v \cdot \bm n\, \left(
			\varphi_N'\left(\frac{\beta_N(g)}{M}\right)(\beta_N(g_w)-\beta_N(g))
			+
			M\varphi_N\left(\frac{\beta_N(g)}{M}\right)
			-
			M\varphi_N\left(\frac{\beta_N(g_w)}{M}\right)
			\right)
			\, \dd\bm v \, \dd\bm s \, \dd t
			\\
			+
			\sum_n
			\sum_{e\in\mathcal{G}^h}
			\left(
			\int\limits_{\mathrlap{I^n \times e \times \mathbb R^3_{\text{out}}}}
			\bm v \cdot \bm n \, M\varphi_N\left(\frac{\beta_N(g)}{M}\right)
			\, \dd\bm v \, \dd\bm s \, \dd t
			+
			\int\limits_{\mathrlap{I^n \times e \times \mathbb R^3_{\text{in}}}}
			\bm v \cdot \bm n \, M\varphi_N\left(\frac{\beta_N(g_w)}{M}\right)
			\, \dd\bm v \, \dd\bm s \, \dd t
			\right) \geq 0,
		\end{multline}
		where the ultimate inequality follows by the convexity of $\varphi_N$
		and the Darrozès--Guiraud inequality from \Cref{lem:Darrozes_Guiraud}.

		To prove the monotonicity of the total entropy, we substitute \eqref{eq:sest} into \eqref{eq:DGFEform} to obtain
		\[
			0\geq s(\bm\lambda^{h,p}; \bm\lambda^{h,p}) - a_{x}^{\Omega}(\bm\lambda^{h,p}; \bm\lambda^{h,p}) - a_{x}^{\partial\Omega}(\bm\lambda^{h,p}; \bm\lambda^{h,p})
			= a_t(\bm\lambda^{h,p}; \bm\lambda^{h,p}).
		\]
		Integrating $a_t$ by parts in $t$, substituting \eqref{eq:Galduality}, applying the reverse chain rule and the divergence theorem on the $I^n$-integral, and adding and subtracting $M\varphi_N(\beta_N(g(t_-^n))/M)$ gives
		\begin{multline}\label{eq:atslabs}
			0 \geq  a_t(\bm\lambda^{h,p}; \bm\lambda^{h,p}) =
			\\
			\sum_n\sum_{\kappa \in \mathcal{T}^h}
			\int\limits_{\mathrlap{\kappa \times \mathbb R^3}}
			\varphi_N'\left(\frac{\beta_N(g(t_+^n))}{M}\right)
			\left(
			\beta_N(g(t_+^{n}))
			-
			\beta_N(g(t_-^{n}))
			\right)
			-
			\background \varphi_N \left(
			\frac{\beta_N(g(t_+^{n}))}{\background}
			\right)
			+
			\background \varphi_N \left(
			\frac{\beta_N(g(t_-^{n}))}{\background}
			\right)
			\, \dd\bm v \, \dd\bm x
			\\
			+
			\sum_n\sum_{\kappa \in \mathcal{T}^h}
			\int\limits_{\mathrlap{\kappa \times \mathbb R^3}}
			\background \varphi_N \left(
			\frac{\beta_N(g(t_-^{n+1}))}{\background}
			\right)
			-
			\background \varphi_N \left(
			\frac{\beta_N(g(t_-^{n}))}{\background}
			\right)
			\, \dd\bm v \, \dd\bm x,
		\end{multline}
		The first summand is nonnegative by the convexity of $\background \varphi_N(\cdot / \background)$; the remaining sum telescopes over the time slabs, yielding \eqref{eq:entstab}, from which the assertion follows.
	\end{proof}

	\section{Details on Mass Flow Through a Channel}
	\label{sec:analytical-mass-flow}
	In this section, we provide the full derivation of the reduced model for the mass flow problem.
	To that end, we seek the distribution in the form
	\[
		f(\bm x, \bm v) = \frac{p(x_3)}{p_0} M_{\rho_0, \bm 0, R_s \theta_0}(\bm v) g(x_1, x_2, \bm v),
	\]
	with an unknown pressure profile $p$, a function $g$ that does not depend on the direction of the flow,
	and the reference density $\rho_0$, temperature $\theta_0$ and pressure $p_0$,
	where $\rho_0 = p_0 / (R_s \theta_0)$.
	Here, $R_s$ denotes the specific gas constant,
	\[
		R_s = \frac{k_B}{m},
	\]
	with the Boltzmann constant $k_B$ and the molecular mass of the gas $m$.
	Inserting this ansatz into the stationary Boltzmann equation yields
	\[
		M_{\rho_0, \bm 0, R_s \theta_0} \big(v_1 \partial_{x_1} g + v_2 \partial_{x_2} g \big)
		+ M_{\rho_0, \bm 0, R_s \theta_0} v_3 \partial_{x_3} \log(p(x_3) / p_0) g
		= \frac{p(x_3)}{p_0} \mathcal C\big( M_{\rho_0, \bm 0, R_s \theta_0} g \big).
	\]
	The pressure term along the channel is nonlinear. However, for the channel geometry $D \ll L$ the nondimensional pressure gradient $X_p = (D/L)(p_1 - p_2)/p_0$ is small, so $p$ stays close to the constant profile $p_0$ and, by the choice of Maxwellian in the ansatz, $g$ stays close to unity. Writing $p = p_0(1 + X_p \varpi)$ and $g = 1 + X_p h$ and expanding to first order in $X_p$, the advection term linearizes, replacing the nonlinear term $\partial_{x_3}\log(p/p_0)\,g$ above by $\partial_{x_3}\varpi$, while we keep the collision operator in its full nonlinear form: our numerical method is designed for the full operator, and since constant functions belong to the discrete trial space, the nonlinear operator already contains its linearization around $M_{\rho_0, \bm 0, R_s \theta_0}$. This partially linearized equation is consistent to first order in $X_p$.

	To fix the pressure-gradient term, we test this equation with $v_3$ and integrate over $\Omega \times \mathbb R^3$: the collision term vanishes by momentum conservation and the advection term integrates to a constant, so $\varpi$ is linear with slope $\partial_{x_3}\varpi = -(p_1 - p_2)/(p_0 L)$ determined by the inlet and outlet pressures. Since $p_1$ and $p_2$ are related by a fixed ratio $p_1/p_2 = q$, the pressure profile has a single degree of freedom, so a constant test function along $x_3$ suffices; integrating along $x_3$ with the midpoint rule replaces $p$ by its midpoint value $p_0 = (p_1 + p_2)/2$, a quadrature whose error is of the same order $O(X_p)$ as the advection linearization. The final equation for $g$ then reads
	\begin{equation}\label{eq:be-channel}
		M_{\rho_0, \bm 0, R_s \theta_0} \big( v_1 \partial_{x_1} g + v_2 \partial_{x_2} g \big)
		- M_{\rho_0, \bm 0, R_s \theta_0} v_3 \frac{1}{L} \frac{p_1 - p_2}{p_0}
		= \mathcal C\big( M_{\rho_0, \bm 0, R_s \theta_0} g \big).
	\end{equation}
	For its nondimensional formulation we set
	\[
		\tilde {\bm x} = \bm x / D, \quad \tilde {\bm v} = \bm v / v_0,
		\quad g(x_1, x_2, \bm v) = \tilde g(x_1 / D, x_2 / D, \bm v/ \sqrt{R_s \theta_0}),
	\]
	with $v_0 = \sqrt{R_s \theta_0}$.
	The resulting equation for $\tilde g$ is
	\begin{equation}\label{eq:be--channel-nondim}
		M_{1, \bm 0, 1} \big(\tilde v_1 \partial_{\tilde x_1} \tilde g
		+ \tilde v_2 \partial_{\tilde x_2} \tilde g \big)
		- X_p M_{1, \bm 0, 1} \tilde v_3 = \frac{1}{\kn} \tilde {\mathcal C}(M_{1, \bm 0, 1} g),
	\end{equation}
	with the Knudsen number
	\begin{equation}
		\label{eq:kn-channel}
		\kn = \sqrt{\frac{\pi}{2}} \frac{1}{D} \frac{\mu}{p_0} \sqrt{R_s \theta_0},
	\end{equation}
	where $\mu$ is the dynamic viscosity of the gas at temperature $\theta_0$,
	and $\tilde {\mathcal C}$ is the nondimensionalized collision operator,
	\[
		\tilde {\mathcal C}(g)(\bm v) = \frac{1}{4 \sqrt{2} \pi} \int_{\mathbb R^3} \int_{\mathbb S^2}
		|\bm v - \bm v_*| \big( g(\acute{\bm v}_*) g(\acute{\bm v}) - g(\bm v_*) g(\bm v) \big)
		\, \dd \bm \sigma
		\, \dd \bm v_*.
	\]
	Our quantity of interest, the mass flow rate through the channel, is computed as
	\[
		m_\text{flx} = \int_\Omega \int_{\mathbb R^3} v_3 f(x_1, x_2, x_3, \bm v)
		\, \dd \bm v \, \dd(x_1, x_2),
	\]
	where, by conservation of mass, the right-hand side is independent of $x_3$.
	To first order in $X_p$, the integral evaluates to
	\[
		\begin{multlined}[t]
			\int_\Omega \int_{\mathbb R^3} v_3 f(x_1, x_2, x_3, \bm v)
			\, \dd \bm v \, \dd(x_1, x_2) \\
			\approx
			\int_\Omega \int_{\mathbb R^3} v_3 M_{\rho_0, \bm 0, R_s \theta_0}(\bm v)
			\, \dd \bm v \, \dd(x_1, x_2)
			+
			X_p \int_\Omega \int_{\mathbb R^3} v_3 \big( \varpi(x_3) + h(x_1, x_2, \bm v) \big)
			M_{\rho_0, \bm 0, R_s \theta_0}(\bm v)
			\, \dd \bm v \, \dd(x_1, x_2).
		\end{multlined}
	\]
	Since the Maxwellian is centered around $\bm 0$, all but one integral vanish,
	and we obtain
	\[
		m_\text{flx} \approx X_p \int_\Omega \int_{\mathbb R^3} v_3 h(x_1, x_2, \bm v)
		M_{\rho_0, \bm 0, R_s \theta_0}(\bm v)
		\, \dd \bm v \, \dd(x_1, x_2).
	\]
	The nondimensional form of the integral reads
	\begin{equation}
		\label{eq:mflux-G-relation}
		m_\text{flx} = X_p \rho_0 v_0 |\Omega| \frac{1}{|\tilde \Omega|} \int_{\tilde \Omega} \int_{\mathbb R^3}
		\tilde v_3 M_{1, \bm 0, 1}(\tilde {\bm v}) \tilde h(\tilde x_1, \tilde x_2, \tilde {\bm v})
		\, \dd \tilde {\bm v} \, \dd(\tilde x_1, \tilde x_2)
		= X_p |\Omega| \rho_0 v_0 G,
	\end{equation}
	with the nondimensional domain $\tilde \Omega$ and the normalized mass flow rate
	\[
		G = \frac{1}{|\tilde \Omega|} \int_{\tilde \Omega} \int_{\mathbb R^3}
		\tilde v_3 M_{1, \bm 0, 1}(\tilde {\bm v}) \tilde h(\tilde x_1, \tilde x_2, \tilde {\bm v})
		\, \dd \tilde {\bm v} \, \dd(\tilde x_1, \tilde x_2).
	\]
	In terms of $\tilde g$, the normalized mass flow rate $G$ reads
	\begin{equation}\label{eq:G-full}
		G = \frac{1}{X_p} \frac{1}{|\tilde \Omega|} \int_{\tilde \Omega} \int_{\mathbb R^3}
		\tilde v_3 M_{1, \bm 0, 1}(\tilde {\bm v}) \tilde g(\tilde x_1, \tilde x_2, \tilde {\bm v})
		\, \dd \tilde {\bm v} \, \dd(\tilde x_1, \tilde x_2).
	\end{equation}
	In the next section, we provide the details on the analytical solution in the collisionless limit.

	\subsection{Analytical Results for the Distribution Function}
	We begin with the derivation of the closed-form expression for the distribution function.
	It exhibits a singularity for small velocities.
	The shape of the distribution function is determined by the two competing
	mechanisms along the pipe, namely the constant pressure gradient, which imposes a nontrivial velocity profile,
	and the diffuse boundary condition that forces the velocity to be zero.
	As our analysis shows, the distribution function depends on the time it takes a single particle
	to travel from a position $(x_1, x_2)$ in the perpendicular plane
	to the boundary.
	For particles of speed close to zero, this can take arbitrarily long,
	thus causing a singularity in the distribution function.
	Nonetheless, the mass flow rate remains finite,
	as the density of such particles in a three-dimensional velocity space
	decays faster than the time to reach the boundary increases.

	\begin{proof}[Proof of \cref{them:pdf-collisionless}]
		Following \eqref{eq:be--channel-nondim}, $g$ in the collisionless limit satisfies
		\[
			M_{1, \bm 0, 1} \big( v_1 \partial_{x_1} g + v_2 \partial_{x_2} g \big) = X_p M_{1, \bm 0, 1} v_3,
		\]
		subject to diffuse reflection boundary condition on $\partial \Omega$ with
		unit temperature and zero velocity.
		With
		\[
			g = 1 + X_p v_3 h,
		\]
		the Boltzmann equation simplifies to
		\[
			v_1 \partial_{x_1} h + v_2 \partial_{x_2} h = 1,
		\]
		with $h = 0$ on the inflow boundary.

		Since $\Omega$ is convex, the characteristic curve
		\[
			X(s) = \bm x^\perp + s \bm v^\perp, \quad s \geq 0,
		\]
		with $\bm v^\perp = (v_1, v_2)$, $\bm x^\perp = (x_1, x_2)$,
		intersects the boundary $\partial \Omega$ at a unique parameter $\tau_\Omega(\bm x^\perp, \bm v^\perp)$.
		Since at the intersection point, the inner product of $\bm v^\perp$ and the normal $\bm n$
		is nonnegative, the parameter for the intersection with the inflow boundary is $\tau_\Omega(\bm x^\perp, - \bm v^\perp)$.
		The function $w(s) = h(X(s), \bm v)$ has $w'(s) = v_1 \partial_{x_1} h + v_2 \partial_{x_2} h = 1$ and $w(\tau_\Omega(\bm x^\perp, - \bm v^\perp)) = 0$, hence $h(\bm x, \bm v) = w(0) = \tau_\Omega(\bm x^\perp, - \bm v^\perp)$ and consequently
		\[
			g(\bm x, \bm v) = 1 + X_p v_3 \tau_\Omega(\bm x^\perp, - \bm v^\perp),
		\]
		thus proving the asserted form of the solution.
	\end{proof}

	In the following, we assume that $\Omega$ is centered around $\bm 0$.
	The travel time $\tau_\Omega$ is homogeneous of degree $-1$ in its second argument, so that it is singular at $\bm v^\perp = \bm 0$. Since $\Omega$ is convex, $\partial(R \Omega) = R \partial \Omega$ for $R > 0$, so $\tau_\Omega$ also scales as $\tau_{R \Omega}(R \bm x^\perp, \bm v^\perp) = R \tau_\Omega(\bm x^\perp, \bm v^\perp)$; that is,
	\[
		\tau_\Omega(\bm x^\perp, \lambda \bm v^\perp)
		= \frac{1}{\lambda} \tau_\Omega(\bm x^\perp, \bm v^\perp), \quad \lambda > 0,
		\qquad
		\tau_{R \Omega}(R \bm x^\perp, R \bm v^\perp) = \tau_{U \Omega}(U \bm x^\perp, U \bm v^\perp) = \tau_\Omega(\bm x^\perp, \bm v^\perp),
	\]
	the combined identities holding for any $R > 0$ and any rotation $U \in \mathbb R^{2 \times 2}$.
	With~\eqref{eq:G-full} we compute the normalized mass flow rate,
	\[
		G = \frac{1}{X_p} \frac{1}{|\Omega|} \int_\Omega \int_{\mathbb R^3} v_3 M_{1, \bm 0, 1}(\bm v) g(\bm x, \bm v)
		\, \dd \bm v \, \dd \bm x
		= \frac{1}{|\Omega|} \int_\Omega \int_{\mathbb R^3} M_{1, \bm 0, 1}(\bm v) v_3^2 \tau_\Omega(\bm x, - \bm v^\perp)
		\, \dd \bm v \, \dd \bm x.
	\]
	Using the homogeneity of $\tau_\Omega$ and the fact that it is independent of $v_3$, the velocity integral simplifies,
	\[
		\begin{multlined}[t]
			\int_{\mathbb R^3} \frac{1}{\sqrt{2 \pi}^3} \exp(-|\bm v|^2 / 2) v_3^2 \tau_\Omega(\bm x, -\bm v^\perp) \, \dd \bm v \\
			= \int_{-\infty}^\infty \frac{v_3^2}{\sqrt{2 \pi}} \exp(-v_3^2 / 2) \, \dd v_3
			\int_0^\infty \frac{r}{2 \pi} \exp(-r^2 / 2) \frac{1}{r} \, \dd r
			\int_0^{2 \pi} \tau_\Omega\big(\bm x, -(\cos \varphi, \sin \varphi)\big) \, \dd \varphi.
		\end{multlined}
	\]
	The integrals over the Maxwellian evaluate to $1 / (2 \sqrt{2 \pi})$, so that
	\begin{equation}\label{eq:G-reduced}
		G = \frac{1}{2 \sqrt{2 \pi}} \frac{1}{|\Omega|} \int_\Omega \int_0^{2 \pi}
		\tau_\Omega\big(\bm z, -(\cos \varphi, \sin \varphi) \big) \, \dd \varphi \, \dd \bm z.
	\end{equation}
	For $\Omega = D_R(\bm 0)$, a disc of radius $R > 0$, the time to the wall starting at $\bm x^\perp$ with velocity $\bm v^\perp$
	is
	\[
		\tau_\Omega(\bm x^\perp, \bm v^\perp) = \frac{1}{|\bm v^\perp|} \Bigg(
		- \bm x^\perp \cdot \frac{\bm v^\perp}{|\bm v^\perp|} + \sqrt{\Big(\bm x^\perp \cdot \frac{\bm v^\perp}{|\bm v^\perp|}\Big)^2 + R^2 - |\bm x^\perp|^2}
		\Bigg).
	\]
	Note that this function has an integrable singularity at $\bm v^\perp = \bm 0$.
	The reduced integral~\eqref{eq:G-reduced} for the normalized mass flux scales linearly with the radius $R$.
	Nondimensionalizing lengths by the channel diameter $D = 2 R$, i.e., setting $R = 1/2$,
	consistent with the coordinates $x_i / D$ used throughout the numerical section, the reduced integral evaluates to $8/3$; hence
	\[
		G = \frac{1}{2 \sqrt{2 \pi}} \cdot \frac{8}{3} = \frac{2 \sqrt{2}}{3 \sqrt{\pi}} \approx 0.531923.
	\]
	The value for the normalized mass flow in the collisionless case in \cref{thm:mass-flow-rate-collisionless}
	was already derived by Knudsen~\cite{knudsen1909} via geometric arguments.
	His cross-sectional mass-flow formula on page~107 of~\cite{knudsen1909}, expressed through
$G = m_\text{flx} / (X_p |\Omega| \rho_0 v_0)$ with normalized pressure gradient
$X_p = 2R(p_1 - p_2)/(p_0 L)$ and reference relation $p_0 = \rho_0 v_0^2$ and then evaluated for a
	circular cross section ($|\Omega|/|\partial\Omega| = R/2$), reduces to the same value
$G = 2 \sqrt{2} / (3 \sqrt{\pi})$, in agreement with \cref{thm:mass-flow-rate-collisionless}.

	\subsection{Post Processing of the Numerical Results}
	\label{sec:post-processing-mass-flow}
	For our simulations of the mass flow problem we fix the ratio $q$ between the
	inlet pressure $p_1$ and the outlet pressure $p_2$,
	\[
		q = 5.
	\]
	Then, we vary $p_1$ from near vacuum to near continuum pressures
	and compute the normalized mass flow rate $G$ in each case.
	For the collisionless limit $p_1 \to 0$, we use our analytical result for $G$.
	For post-processing and visualization, we use the nondimensional rarefaction parameter,
	\[
		\delta = \frac{\sqrt{\pi}}{2} \kn^{-1} = \frac{D}{\sqrt{2 R_s \theta_0}} \frac{p_0}{\mu}
		= \frac{D}{\sqrt{2 R_s \theta_0}} \frac{1 + q}{2 \mu} p_1,
	\]
	which is proportional to the inlet pressure.
	Owing to the wide range of rarefaction parameters involved,
	we present the results on a logarithmic scale,
	\[
		\tilde \delta = \log(1 + \delta), \quad \tilde G = \log G.
	\]
	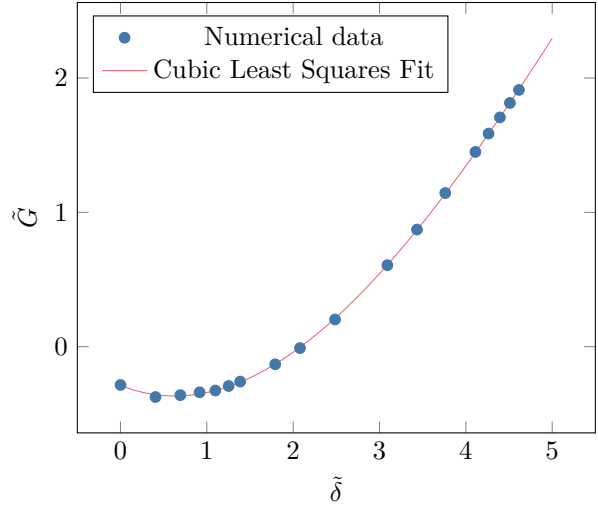
\begin{figure}
		\centering
		\begin{tikzpicture}
			\begin{axis}[
					xlabel = $\tilde \delta$,
					ylabel = $\tilde G$,
					legend pos = north west,
				]
				\addplot+[only marks] table {data/G-raw-log-mom.dat};
				\addlegendentry{Numerical data}
				\addplot+[domain=0:5, no marks] {-0.288547 - 0.252736 * x + 0.211029 * x^2 - 0.0114248 * x^3};
				\addlegendentry{Cubic Least Squares Fit}
			\end{axis}
		\end{tikzpicture}
		\caption{Numerical results of normalized mass flow rates in rescaled coordinates
			with least squares fit $\tilde G(\tilde \delta)$.}
		\label{fig:raw-G-log}
	\end{figure}
	\Cref{fig:raw-G-log} displays the rescaled numerical results
	together with a least squares fit of a cubic polynomial.
	The polynomial approximates our results very well over the full range
	of simulated rarefaction parameters.
	The normalized mass flow rate in the full range is computed by
	\[
		G(\delta) = \exp\bigg( \tilde G\big( \log(1 + \delta) \big) \bigg),
		\quad \delta \geq 0,
	\]
	where $\tilde G$ denotes our cubic least squares fit of the data in \cref{fig:raw-G-log},
	\[
		\tilde G(\tilde \delta) = -0.63488 - 0.253103 \tilde \delta + 0.211179 \tilde \delta^2 - 0.0114429 \tilde \delta^3,
	\]
	with residual $1 - R^2 = \num{5.64e-5}$.
\bibliography{book}

\end{document}